\documentclass[11pt]{article}

\usepackage{mathpazo}
\usepackage{mathdots}
\usepackage{amsfonts}
\usepackage{amsmath}
\usepackage{graphicx}
\usepackage{latexsym}
\usepackage{mathabx}
\usepackage{bm}
\usepackage{longtable}
\usepackage{booktabs}
\usepackage{makecell}

\usepackage[margin=1.05in]{geometry}
  
\usepackage[T1]{fontenc}
\usepackage{times}
\usepackage{color,graphicx}
\usepackage{array}
\usepackage{enumerate}
\usepackage{amsmath}
\usepackage{amssymb}
\usepackage{amsthm}
\usepackage{pgfplots}
\usepackage{pgf}
\usepackage{tikz}
\usetikzlibrary{patterns}
\usepgfplotslibrary{patchplots} 
\usetikzlibrary{decorations.markings}
\usetikzlibrary{pgfplots.patchplots} 
\pgfplotsset{width=9cm,compat=1.12}
    \usepgfplotslibrary{colormaps,external}
    \definecolor{ocre}{RGB}{0,96,128}
    \definecolor{purp}{RGB}{112,0,112}
    \pgfplotsset{
    	colormap={ocrefade}{
    		rgb255=(150,216,255)
    		rgb255=(0,64,96)
    	}
    }

\usepackage{amsthm}
\newtheorem{conjecture}{Conjecture}
\newtheorem{definition}{Definition}
\newtheorem{lemma}{Lemma}

\newtheorem{theorem}{Theorem}
\newtheorem{corollary}{Corollary}

\newtheorem{fact}{Fact}
\newtheorem{proposition}{Proposition}

\usepackage{xcolor}
\usepackage{makeidx}
\usepackage[colorlinks=true,linkcolor=blue,anchorcolor=blue,citecolor=red,urlcolor=magenta]{hyperref}
\usepackage{caption}
\def\C{\mathbb{C}}

\def\R{\mathbb{R}}
\def\Z{\mathbb{Z}}

\begin{document}
\title{Laplacian $\{-1,0,1\}$- and $\{-1,1\}$-diagonalizable graphs}

\author{
	Nathaniel Johnston\textsuperscript{1} and Sarah Plosker\textsuperscript{2}
}

\maketitle

\begin{abstract}
    A graph is called \emph{Laplacian integral} if the eigenvalues of its Laplacian matrix are all integers. We investigate the subset of these graphs whose Laplacian is furthermore diagonalized by a matrix with entries coming from a fixed set, in particular, the sets $\{-1,0,1\}$ or $\{-1,1\}$. Such graphs include as special cases the recently-investigated families of \emph{Hadamard-diagonalizable} and \emph{weakly Hadamard-diagonalizable} graphs. As a combinatorial tool to aid in our investigation, we introduce a family of vectors that we call \emph{balanced}, which generalizes totally balanced partitions, regular sequences, and complete partitions. We show that balanced vectors completely characterize which graph complements and complete multipartite graphs are $\{-1,0,1\}$-diagonalizable, and we furthermore prove results on diagonalizability of the Cartesian product, disjoint union, and join of graphs. Particular attention is paid to the $\{-1,0,1\}$- and $\{-1,1\}$-diagonalizability of the complete graphs and complete multipartite graphs. Finally, we provide a complete list of all simple, connected graphs on nine or fewer vertices that are $\{-1,0,1\}$- or $\{-1,1\}$-diagonalizable.
    \medskip

	
\end{abstract}

\addtocounter{footnote}{1}
\footnotetext{Department of Mathematics \& Computer Science, Mount Allison University, Sackville, NB, Canada E4L 1E4}
\addtocounter{footnote}{1}
\footnotetext{Department of Mathematics \& Computer Science, Brandon University, Brandon, MB, Canada R7A 6A9}

\section{Introduction}\label{sec:intro}

In spectral graph theory, the goal is to identify how the eigenvalues and/or eigenvectors of a graph's adjacency and/or Laplacian matrix relate to properties of the graph itself \cite{BH11}. In recent years, a significant amount of attention in this area has been given to graphs with eigenspaces that are spanned by eigenvectors with restricted entries. For example, numerous papers have explored eigenspaces of graphs that are spanned entirely by vectors whose entries come from the set $\{-1,0,1\}$ \cite{AAG06,San08}, while others have explored graphs that are diagonalizable by a Hadamard matrix (and thus have eigenspaces spanned entirely by vectors whose entries come from the set $\{-1,1\}$) \cite{BFK11,Breen22}. Sometimes orthogonality is expected between these eigenvectors, while other times only a weaker condition like linear independence or quasi-orthogonality \cite{adm2021weakly} is required.

One of the prototypical applications of such graphs arises in quantum information theory. In this setting, the concept of ``perfect state transfer'' (see \cite{ADLMTXZ16} and references therein) describes the ability to transfer a quantum state from one vertex of a graph to another reliably, which is required in order for quantum computers to function properly \cite{Bose}. Hadamard diagonalizability of a graph's Laplacian has been shown to be a useful tool when investigating whether or not it exhibits perfect state transfer \cite{johnston2017}, and more recently perfect state transfer in graphs whose Laplacian  can be diagonalized by a weak Hadamard matrix has been studied in \cite{MMP}.

In this paper, we explore which graphs have a Laplacian matrix that can be diagonalized by a $\{-1,0,1\}$- or $\{-1,1\}$-valued matrix. Unlike recent work on Hadamard-diagonalizable graphs and weakly Hadamard-diagonalizable graphs, we do not enforce any orthogonality or quasi-orthogonality conditions on the $\{-1,0,1\}$- or $\{-1,1\}$-valued matrices. Instead, we introduce the concepts of $\{-1,0,1\}$-bandwidth and $\{-1,1\}$-bandwidth of a graph, which generalize and unify these different notions of orthogonality. For example, a graph is Hadamard-diagonalizable if and only if it has $\{-1,1\}$-bandwidth $1$ and it is weakly Hadamard-diagonalizable if and only if it has $\{-1,0,1\}$-bandwidth at most $2$ (see Section~\ref{sec:bandwidth} for details).

As a combinatorial tool that we believe is of independent interest, we introduce a family of vectors that we call \emph{balanced}. These vectors generalize numerous other combinatorial objects that have been studied previously:
\begin{itemize}
    \item recursively balanced partitions, which were introduced for their connections to weakly Hadamard diagonalizable graphs \cite{adm2021weakly};

    \item regular sequences, which have applications particularly in statistics \cite{FO89}; and

    \item complete partitions, which are partitions of a given number that contain, as subsets, partitions of all smaller numbers \cite{Par98}.
\end{itemize}
We show that balanced vectors are at the heart of $\{-1,0,1\}$-diagonalizability, as they completely determine whether or not the complement of a $\{-1,0,1\}$-diagonalizable graph is $\{-1,0,1\}$-diagonalizable (see Theorem~\ref{thm:complement_of_neg_zero_one}) and which complete multipartite graphs are $\{-1,0,1\}$-diagonalizable (see Corollary~\ref{cor:complete_multipartite_graph}).

\subsection{Organization of the paper}\label{sec:organization}

We start in Section~\ref{sec:prelim} by presenting the necessary mathematical preliminaries and definitions,  
and introducing the notation used throughout the paper. In Section~\ref{sec:balanced_vecs}, we define and explore the concept of balanced vectors. 
In Section~\ref{sec:wHdiagGraphs}, we study $\{-1,0,1\}$- and $\{-1,1\}$-diagonalizability of complete graphs, complete multipartite graphs, complements of $\{-1,0,1\}$- and $\{-1,1\}$-diagonalizable graphs, and joins of $\{-1,0,1\}$- and $\{-1,1\}$-diagonalizable graphs. In particular, we show that this type of diagonalizability is closely connected with the well-known Hadamard conjecture, a natural $\{-1,0,1\}$-valued variant of it, and the balanced vectors that were introduced earlier. We furthermore provide, in Section~\ref{sec:small_graphs}, an exhaustive list of all simple graphs on nine or fewer vertices that are $\{-1,0,1\}$-diagonalizable, as well their $\{-1,0,1\}$- and $\{-1,1\}$-bandwidths. 
In Section~\ref{sec:other_set_S} we generalize our results and definitions to sets other than $\{-1,0,1\}$ and $\{-1,1\}$, and finally we close in Section~\ref{sec:conclusions} with some open questions.

\section{Preliminaries and definitions}\label{sec:prelim}

We use bold lowercase letters like $\mathbf{v}, \mathbf{w}, \ldots$ to denote vectors in $\R^n$ or $\C^n$, $\mathbf{e_j}$ to denote the $j$-th standard basis vector (i.e., the vector with $1$ in its $j$-th entry and $0$ elsewhere), and $\mathbf{1}_n$ to denote the $n$-entry all-ones vector (or just $\mathbf{1}$ if the number of entries is unimportant or clear from context). We use $\Z^p_{+}$ to denote vectors of length $p\in \mathbb \Z_{+}$ having all (strictly) positive integer entries. We denote the set of all $m\times n$ real-valued matrices by $\mathcal{M}_{m,n}$, or simply $\mathcal{M}_n$ in the case when $m=n$, and we use upper case letters like $A, B, \ldots$ to denote matrices. The $n \times n$ identity matrix is denoted by $I_{n}$ (or simply $I$ if its size is unimportant or clear from context) and the all-ones matrix is denoted by $J$.

A superscript $T$ denotes the transpose of a vector or matrix, although we will often freely associate between vectors viewed as columns and vectors viewed as rows. We denote the set of $n \times n$ positive semi-definite (PSD) real matrices by $\mathcal{M}_n^+$ and we require symmetry of all PSD matrices. We use non-bold lowercase letters with subscripts to denote specific entries of vectors and matrices: $v_j$ is the $j$-th entry of $\mathbf{v}$ and $a_{i,j}$ is the $(i,j)$-entry of $A$.

We exclusively consider simple graphs: unweighted, undirected graphs with no loops. If $G$ is a simple graph on $n$ vertices then its \emph{adjacency matrix} is the symmetric matrix $A := [a_{i,j}] \in \mathcal{M}_n$ that has $a_{i,j} = 1$ if there is an edge between vertices $i$ and $j$ and $a_{i,j} = 0$ otherwise. The \emph{degree matrix} $D \in \mathcal{M}_n$ is a diagonal matrix whose diagonal entry $d_{i,i}$ ($1 \leq i \leq n$) equals the \emph{degree} of the vertex $i$ (i.e., the number of edges incident to vertex $i$). The graph $G$ is \emph{regular} if $d_{i,i}$ is constant for all $i$. The \emph{Laplacian matrix} of $G$ is the symmetric matrix $L := D - A$.
 
The Laplacian matrix $L$ of a simple graph $G$ has numerous useful properties:
\begin{itemize}
    \item $L$ is diagonally dominant and thus positive semidefinite;

    \item $L$ has row sums equal to $0$. In other words, if $\mathbf{0},\mathbf{1} \in \R^n$ are the all-zeros and all-ones vectors, respectively, then $L\mathbf{1} = \mathbf{0}$, so $\mathbf{1}$ is an eigenvector of $L$ with corresponding eigenvalue $0$; and

    \item the multiplicity of the eigenvalue $0$ is exactly the number of connected components of $G$. In particular, $G$ is connected if and only if the eigenspace corresponding to the eigenvalue $0$ is $\mathrm{span}(\mathbf{1})$.
\end{itemize}

\subsection{Laplacian integral and Hadamard-diagonalizable graphs}\label{sec:laplacian_integral}

A graph is called \emph{Laplacian integral} if all of the eigenvalues of its Laplacian matrix are integers (see \cite{Grone94,Grone08,Kirk} and the references therein). Numerous important families of graphs are Laplacian integral, including all cographs \cite{Mer98} (i.e., graphs generated from a single vertex along with the operations of graph complement and disjoint union) and thus their sub-families of complete multipartite graphs \cite{zhao} and threshold graphs. If a simple graph is Laplacian integral then its eigenvectors can be chosen to have integer entries as well, since the entries of $L-\lambda I$ are all integers and thus the equation $(L-\lambda I)\mathbf{x}=\mathbf{0}$ has a solution with $\mathbf{x} \in \mathbb{Z}^n$.

We are often interested in which Laplacian integral graphs have eigenvectors whose entries can be further restricted to belong to the set $\{-1,0,1\}$ or even $\{-1,1\}$:

\begin{definition}\label{defn:S_diag}
    We say that a graph $G$ is \emph{$\{-1,0,1\}$-diagonalizable} (\emph{$\{-1,1\}$-diagonalizable}) if there is a basis of eigenvectors for the Laplacian $L$ of $G$ whose entries all belong to $\{-1,0,1\}$ ($\{-1,1\}$). Equivalently, $G$ is $\{-1,0,1\}$-diagonalizable ($\{-1,1\}$-diagonalizable) if there exists a matrix $P$, whose entries all belong to $\{-1,0,1\}$ ($\{-1,1\}$), with the property that $P^{-1}LP$ is diagonal.
\end{definition}

Some easily-provable properties of $\{-1,0,1\}$- and $\{-1,1\}$-diagonalizable graphs include:
\begin{itemize}
    \item If $G$ is a simple $\{-1,1\}$-diagonalizable graph then all eigenvalues of its Laplacian matrix $L$ are even integers. This can be proved via an argument that is identical to that of \cite[Theorem~5]{BFK11}.
    
    \item If $G$ is a simple connected $\{-1,0,1\}$-diagonalizable graph then orthogonality of the eigenspaces of real symmetric matrices implies that all $\{-1,0,1\}$-eigenvectors of $L$ other than $\pm \mathbf{1}$ have the same number of entries equal to $1$ as $-1$.
\end{itemize}

The set of graphs that are $\{-1,1\}$-diagonalizable contains the set of Hadamard-diagonalizable graphs as a subset: these are graphs $G$ whose Laplacian matrix $L$ can be diagonalized by a Hadamard matrix $H$ (i.e., a matrix whose entries all belong to $\{-1,1\}$ and whose columns are mutually orthogonal). Similarly, graphs that are $\{-1,0,1\}$-diagonalizable contain, as a sub-family, graphs that are weakly Hamadard-diagonalizable \cite{adm2021weakly}:

\begin{definition}\label{defn:weak_hadamard}
    A matrix $P \in \mathcal{M}_n$ is a \emph{weak Hadamard matrix} if all of its entries belong to $\{-1,0,1\}$ and $P^TP$ is a tridiagonal matrix. A graph $G$ is \emph{weakly Hadamard-diagonalizable (WHD)} if its Laplacian matrix $L$ can be diagonalized by a weak Hadamard matrix (i.e., $P^{-1}LP$ is diagonal for some weak Hadamard matrix $P$).
\end{definition}

\subsection{Bandwidth and (quasi-)orthogonality}\label{sec:bandwidth}

The tridiagonal requirement of weak Hadamard matrices from Definition~\ref{defn:weak_hadamard} relaxes the requirement of (non-weak) Hadamard matrices that their columns be mutually orthogonal. In particular, $P^TP$ being tridiagonal is equivalent to the columns $\mathbf{v_1}$, $\mathbf{v_2}$, $\ldots$ , $\mathbf{v_n}$ of $P$ being \emph{quasi-orthogonal}: $\mathbf{v_i} \cdot \mathbf{v_j} = 0$ whenever $|i-j| \geq 2$.

Recall that the \emph{bandwidth} of a matrix $A$ is the smallest $k \in \mathbb{Z}_+$ with the property that $a_{i,j} = 0$ whenever $|i-j| \geq k$ (so, for example, diagonal matrices have bandwidth $1$ and tridiagonal matrices have bandwidth $2$) \cite[Chapter 7.2]{bandwidth}.\footnote{Some sources define the bandwidth of a matrix to be $1$ less than the definition used herein, so that diagonal and tridiagonal matrices have bandwidth $0$ and $1$, respectively.} It is then natural to define the bandwidth of a graph's diagonalization as follows:

\begin{definition}\label{def:bandwidth}
    Let $G$ be a graph with Laplacian matrix $L$. 
    \begin{enumerate}
        \item[a)] If $G$ is not $\{-1,0,1\}$-diagonalizable ($\{-1,1\}$-diagonalizable) then the \emph{$\{-1,0,1\}$-bandwidth} (\emph{$\{-1,1\}$-bandwidth}) of $G$ is $\infty$.
        
        \item[b)] Otherwise, the \emph{$\{-1,0,1\}$-bandwidth} (\emph{$\{-1,1\}$-bandwidth}) of $G$ is the smallest $k \in \mathbb{Z}_+$ for which there exists a matrix $P$ with all of the following properties:
        \begin{itemize}
            \item all of its entries belong to $\{-1,0,1\}$ ($\{-1,1\}$);

            \item $P^{-1}LP$ is diagonal; and

            \item $P^TP$ has bandwidth $k$.
        \end{itemize}
    \end{enumerate}
\end{definition}

It is immediate from this definition that if an $n$-vertex graph has finite $\{-1,0,1\}$- or $\{-1,1\}$-bandwidth then that bandwidth is at most $n$. In fact, we can do slightly better by recalling that if a graph has $c \geq 1$ connected components then its Laplacian matrix has at least two distinct eigenvalues ($0$ and a strictly positive number), and the corresponding eigenspaces are orthogonal to each other. It follows that the $\{-1,0,1\}$- and $\{-1,1\}$-bandwidths of the graph are at most $\max\{c,n-c\}$. We will see examples shortly of graphs whose bandwidths attain this bound, at least for certain choices of $c$ and $n$.

It is also immediate from Definition~\ref{def:bandwidth} that a graph has $\{-1,1\}$-bandwidth equal to $1$ if and only if it is Hadamard diagonalizable, and it has $\{-1,0,1\}$-bandwidth at most $2$ if and only if it is weakly Hadamard diagonalizable. Many of the results about weakly Hadamard diagonalizable graphs from \cite{adm2021weakly} still hold for arbitrary $\{-1,0,1\}$- and $\{-1,1\}$-bandwidths. Indeed, many of the proofs from that paper do not make use of quasi-orthogonality of the Laplacian's eigenvectors, so the same proofs establish the following facts:

\begin{proposition}\label{prop:neg101properties}
    Let $G$ and $H$ be graphs with $m$ and $n$ vertices, respectively.
   \begin{enumerate}
       \item[a)] If $G$ is connected and $\{-1,0,1\}$-diagonalizable ($\{-1,1\}$-diagonalizable) with bandwidth $k$ then its complement $G^{\textup{c}}$ is $\{-1,0,1\}$-diagonalizable ($\{-1,1\}$-diagonalizable) with bandwidth at most $k$ (c.f.\ \cite[Lemma 2.4]{adm2021weakly}).
       
       \item[b)] If $G$ and $H$ are $\{-1,0,1\}$-diagonalizable with bandwidths $k$ and $\ell$, respectively, then the disjoint union $G \sqcup H$ is $\{-1,0,1\}$-diagonalizable with bandwidth $\max\{k,\ell\}$ (c.f.\ \cite[Lemma 2.3]{adm2021weakly}).
       
       \item[c)] If $G$ and $H$ are $\{-1,0,1\}$-diagonalizable ($\{-1,1\}$-diagonalizable) with bandwidths $k$ and $\ell$, respectively, then their Cartesian product $G \mathbin{\square} H$ is $\{-1,0,1\}$-diagonalizable ($\{-1,1\}$-diagonalizable) with bandwidth at most $\min\{n(k-1)+\ell,m(\ell-1)+k\}$ (c.f.\ \cite[Proposition 3.3]{adm2021weakly}).
   \end{enumerate} 
\end{proposition}

We do not prove this proposition, since we present a more general version of it later (as Proposition~\ref{prop:neg101propertiesS}) that we do prove.

There are some technicalities concerning each part of the above proposition that we should be careful to clarify. It is tempting to conjecture that the bandwidth does not change when taking the complement in property~(a), but it actually can strictly decrease. For example, the $\{-1,0,1\}$-bandwidth of the (connected) complete multipartite graph $K_{4,1,1,1,1,1}$ is $2$, but the $\{-1,0,1\}$-bandwidth of its (disconnected) complement $K_{4,1,1,1,1,1}^{\textup{c}} = K_4 \sqcup K_1 \sqcup K_1 \sqcup K_1 \sqcup K_1 \sqcup K_1$ is $1$. Furthermore, property~(a) relies crucially on $G$ being connected: if it is disconnected then its complement might not be $\{-1,0,1\}$-diagonalizable. For example, $K_2 \sqcup K_1$ is $\{-1,0,1\}$-diagonalizable, but its complement $(K_2 \sqcup K_1)^{\textup{c}} = K_{2,1}$ is not. We return to this problem in Section~\ref{sec:complete_multipartite}.

The fact that property~(b) does not mention $\{-1,1\}$-diagonalizability is not an accident: the disjoint union preserves $\{-1,0,1\}$-diagonalizability but not $\{-1,1\}$-diagonalizability. For example, $K_2$ and $K_1$ are each $\{-1,1\}$-diagonalizable but $K_2 \sqcup K_1$ is not.

In property~(c), the strange quantity $\min\{n(k-1)+\ell,m(\ell-1)+k\}$ comes from the fact that if $A \in \mathcal{M}_m$ and $B \in \mathcal{M}_n$ are matrices with bandwidth $k$ and $\ell$, respectively, then $A \otimes B$ has bandwidth at most $n(k-1)+\ell$ and $B \otimes A$ has bandwidth at most $m(\ell-1)+k$.

If $k = 1$ and $\ell = 2$ then property~(c) says that $G \mathbin{\square} H$ has $\{-1,0,1\}$-bandwidth at most $2$ (i.e., $G \mathbin{\square} H$ is weakly Hadamard diagonalizable), thus recovering one of the statements of \cite[Proposition~3.3]{adm2021weakly}. The other two statements of that proposition (i.e., that the direct product and strong product of $G$ and $H$ are also weakly Hadamard diagonalizable) are actually false: $K_{2,1,1}$ has $\{-1,0,1\}$-bandwidth equal to $1$, but its strong and direct products with itself are not even Laplacian integral, let alone $\{-1,0,1\}$-diagonalizable or weakly Hadamard diagonalizable.\footnote{The statements of \cite[Proposition~3.3]{adm2021weakly} concerning the strong and direct products are true if the hypothesis ``regular'' is added to the statement of the proposition, which is likely what was intended.}

\section{Balanced vectors}\label{sec:balanced_vecs}

Here, we introduce the concept of balanced vectors, that will let us characterize $\{-1,0,1\}$-diagonalizability of certain families of graphs:

\begin{definition}\label{defn:totally_balanced}
    A vector $\mathbf{v} \in \R^p$ ($p \geq 2$) is \emph{balanced} if there exists a matrix $A \in \mathcal{M}_{p-1,p}$, all of whose entries belong to $\{-1,0,1\}$, with $\mathrm{null}(A) = \mathrm{span}(\mathbf{v})$.
\end{definition}
 
A vector $\mathbf{v}$ being balanced is equivalent, in the language of \cite[Proposition~4.2]{adm2021weakly},\footnote{In particular, this means that balanced vectors generalize \emph{recursively balanced partitions} from that paper. The fact that $(1,1,1)$ is balanced but not recursively balanced shows that this generalization is strict.} to there being $p-1$ equations of the form
\[
    v_{i_1} + v_{i_2} + \cdots + v_{i_r} = v_{j_1} + v_{j_2} + \cdots + v_{j_s}.
\]
These $p-1$ equations can be obtained by letting $i_1$, $i_2$, $\ldots$ , $i_r$ be the indices of the ``$-1$'' entries in a particular row of the matrix $A$, and letting $j_1$, $j_2$, $\ldots$ , $j_s$ be the indices the of the ``$+1$'' entries in that row. That is, for each row of $A$, the sum of the entries of $\mathbf{v}$ corresponding to the $-1$ entries in a particular row of $A$ is equal to the sum of the entries of $\mathbf{v}$ corresponding to the $1$ entries in that row. 

Implicit in Definition~\ref{defn:totally_balanced} is that the rank of $A$ must be $p-1$, its nullity must be $1$, and balanced vectors are all non-zero. We can also, without loss of generality, assume many other nice properties of balanced vectors. For example, whether or not a vector is balanced does not depend on the order of its entries, since any permutation that we apply to the entries of $\mathbf{v}$ can also be applied to the columns of the matrix $A$ from Definition~\ref{defn:totally_balanced}. We may thus without loss of generality assume that their entries are sorted in non-increasing order. Furthermore, the following proposition shows that we may assume without loss of generality that the entries of a balanced vector are integers:

\begin{proposition}\label{prop:tot_bal_rescale}
    Let $\mathbf{v} \in \R^p$ be a balanced vector. Then $c\mathbf{v}$ is also balanced for all non-zero $c \in \R$. Furthermore, there exists $c \in \R$ so that $c\mathbf{v} \in \Z^p$.
\end{proposition}

\begin{proof}
    The fact that $c\mathbf{v}$ is also balanced follows from the fact that the null space of every matrix is a subspace, so $\mathbf{v} \in \mathrm{null}(A)$ implies $c\mathbf{v} \in \mathrm{null}(A)$. Furthermore, since $A$ has full rank $p-1$, each row of the row reduced echelon form $R$ of $A$ has $p-2$ entries equal to $0$, one entry equal to $1$, and one entry that is some rational number (potentially equal to $0$ or $1$). It follows that we can find a non-zero vector in $\mathrm{null}(A) = \mathrm{null}(R)$ with all entries integers no larger than the least common multiple of the denominators of entries of $R$.
\end{proof}

We can also assume without loss of generality that the entries of a balanced vector $\mathbf{v}$ are all strictly positive. To see this, notice that multiplying the $j$-th column of the matrix $A$ from Definition~\ref{defn:totally_balanced} by $-1$ results in a new balanced vector with its $j$-th entry multiplied by $-1$. Similarly, adding or removing a ``$0$'' entry to or from a balanced vector results in another balanced vector; simply add or remove a single row and column to or from the matrix $A$ from Definition~\ref{defn:totally_balanced} in the obvious way.\footnote{To make this operation still preserve balancedness of vectors when $p = 2$, we extend Definition~\ref{defn:totally_balanced} to the $p = 1$ case by simply specifying that $v \in \mathbb{R}^1$ is balanced if and only if it is non-zero.} We summarize all of these ``without loss of generality'' observations in the following fact:

\begin{fact}\label{fact:totally_balanced_wlog}
    The following operations all preserve balancedness of a vector $\mathbf{v} \in \mathbb{R}^p$:
    \begin{itemize}
        \item permuting the entries of $\mathbf{v}$;

        \item replacing $\mathbf{v}$ by $c\mathbf{v}$ for some $0 \neq c \in \mathbb{R}$;

        \item changing the sign of some entries of $\mathbf{v}$; and
        
        \item adding or removing $r$ number of ``$0$'' entries to $\mathbf{v}$, in which case the resulting vector is in $\mathbb R^{p\pm r}$.
    \end{itemize}
    Furthermore, every balanced vector in $\mathbb{R}^p$ can be obtained via these operations from a non-increasing positive integer-valued balanced vector (i.e., a balanced vector with integer entries satisfying $v_1 \geq v_2 \geq \cdots \geq v_p \geq 1$).
\end{fact}

Because of the above fact, we typically only consider non-increasing balanced vectors in $\Z^p_{+}$ whose entries have greatest common divisor equal to $1$. For example, the only (up to scaling and re-ordering of entries) balanced vectors in $\Z^3_{+}$ are $(1,1,1)$ and $(2,1,1)$, which come from the null spaces of the $2 \times 3$ matrices
\[
    \begin{bmatrix}
        1 & -1 & 0 \\
        0 & 1 & -1
    \end{bmatrix} \quad \text{and} \quad \begin{bmatrix}
        1 & -1 & -1 \\
        0 & 1 & -1
    \end{bmatrix},
\]
respectively. Every other balanced vector in $\mathbb{R}^3$ can be obtained by applying some combination of the first three operations from Fact~\ref{fact:totally_balanced_wlog} to one of these two vectors, or by adding a ``$0$'' entry to a balanced vector from $\mathbb{R}^2$.

The following proposition gives another way to construct new balanced vectors from old ones:

\begin{proposition}\label{prop:concat_tb}
    Suppose $\mathbf{v} \in \R^p$ and $\mathbf{w} \in \R^q$ are balanced. If there exist vectors $\mathbf{x} \in \{-1,0,1\}^p$ and $\mathbf{y} \in \{-1,0,1\}^q$ with $\mathbf{v} \cdot \mathbf{x} = \mathbf{w} \cdot \mathbf{y} \neq 0$, then $(\mathbf{v},\mathbf{w})\in \mathbb R^{p+q}$ is balanced.
\end{proposition}

Before proving the above proposition, we note that the curious hypothesis involving the dot product is indeed required for the conclusion to hold. For example, both $\mathbf{v} = (3,3)$ and $\mathbf{w} = (1,1)$ are balanced, but $(\mathbf{v},\mathbf{w}) = (3,3,1,1)$ is not. The dot product hypothesis ensures that $\mathbf{v}$ and $\mathbf{w}$ have scaling that is not ``too far off'' from each other (and since re-scaling a balanced vector produces another balanced vector, this is not too much of a restriction). For example, if we choose $\mathbf{x}$ and $\mathbf{y}$ to be standard basis vectors in Proposition~\ref{prop:concat_tb} then we learn the following: if $\mathbf{v} \in \Z^p_{+}$ and $\mathbf{w} \in \Z^q_{+}$ are balanced and $v_i = w_j$ for some $i,j$, then $(\mathbf{v},\mathbf{w})$ is balanced.

\begin{proof}[Proof of Proposition~\ref{prop:concat_tb}]
    Since $\mathbf{v}$ and $\mathbf{w}$ are balanced, there exist matrices $A \in \mathcal{M}_{p-1,p}$ and $B \in \mathcal{M}_{q-1,q}$ with $\mathrm{null}(A) = \mathrm{span}(\mathbf{v})$ and $\mathrm{null}(B) = \mathrm{span}(\mathbf{w})$. If $\mathbf{x}$ and $\mathbf{y}$ are as in the statement of the proposition then it is straightforward to verify that when
    \[
        C := \begin{bmatrix}
            A & O \\
            O & B \\
            \mathbf{x}^T & -\mathbf{y}^T
        \end{bmatrix} \in \mathcal{M}_{p+q-1,p+q},
    \]
    we have $\mathrm{null}(C) = \mathrm{span}(\mathbf{v},\mathbf{w})$, so $(\mathbf{v},\mathbf{w})$ is balanced.
\end{proof}

We provide a complete list (up to scaling and re-ordering of entries) of balanced vectors in $\mathbb{Z}^p_{+}$ for $p \in \{2,3,4,5\}$ in Table~\ref{tab:totally_balanced_smallp}.

\begin{table}[!htb]
    \centering
    \begin{tabular}{cl}\toprule
        $p$ & vectors $\mathbf{v}$ \\\toprule
        $1$ & $(1)$ \\\midrule
        $2$ & $(1,1)$ \\\midrule
        $3$ & $(1,1,1)$, $(2,1,1)$ \\\midrule
        $4$ & $(1,1,1,1)$, $(2,1,1,1)$, $(2,2,1,1)$, $(3,1,1,1)$, $(3,2,1,1)$, $(3,2,2,1)$, $(4,2,1,1)$, $(4,3,2,1)$ \\\midrule
        $5$ & $(1,1,1,1,1)$, $(2,1,1,1,1)$, $(2,2,1,1,1)$, $(2,2,2,1,1)$, $(3,1,1,1,1)$, $(3,2,1,1,1)$, $(3,2,2,1,1)$, \\
        & $(3,2,2,2,1)$, $(3,3,1,1,1)$, $(3,3,2,1,1)$, $(3,3,2,2,1)$, $(3,3,2,2,2)$, $(4,1,1,1,1)$, $(4,2,1,1,1)$, \\
        & $(4,2,2,1,1)$, $(4,3,1,1,1)$, $(4,3,2,1,1)$, $(4,3,2,2,1)$, $(4,3,3,1,1)$, $(4,3,3,2,1)$, $(4,3,3,2,2)$, \\
        & $(4,4,2,1,1)$, $(4,4,3,2,1)$, $(4,4,3,3,2)$, $(5,2,1,1,1)$, $(5,2,2,1,1)$, $(5,2,2,2,1)$, $(5,3,1,1,1)$, \\
        & $(5,3,2,1,1)$, $(5,3,2,2,1)$, $(5,3,3,1,1)$, $(5,3,3,2,1)$, $(5,4,2,1,1)$, $(5,4,2,2,1)$, $(5,4,3,1,1)$, \\
        & $(5,4,3,2,1)$, $(5,4,3,2,2)$, $(5,4,3,3,1)$, $(5,4,3,3,2)$, $(5,4,4,2,1)$, $(5,4,4,3,2)$, $(6,2,2,1,1)$, \\
        & $(6,3,1,1,1)$, $(6,3,2,1,1)$, $(6,3,2,2,1)$, $(6,4,2,1,1)$, $(6,4,3,1,1)$, $(6,4,3,2,1)$, $(6,4,4,1,1)$, \\
        & $(6,4,4,3,1)$, $(6,5,2,2,1)$, $(6,5,3,1,1)$, $(6,5,3,2,1)$, $(6,5,3,2,2)$, $(6,5,4,2,1)$, $(6,5,4,3,1)$, \\
        & $(6,5,4,3,2)$, $(6,5,4,4,3)$, $(7,3,2,1,1)$, $(7,3,2,2,1)$, $(7,3,3,2,1)$, $(7,4,2,1,1)$, $(7,4,2,2,1)$, \\
        & $(7,4,3,2,1)$, $(7,4,4,2,1)$, $(7,5,3,1,1)$, $(7,5,3,2,1)$, $(7,5,3,3,1)$, $(7,5,4,1,1)$, $(7,5,4,2,1)$, \\
        & $(7,5,4,3,1)$, $(7,6,3,2,1)$, $(7,6,3,2,2)$, $(7,6,4,2,1)$, $(7,6,4,3,2)$, $(7,6,5,3,1)$, $(7,6,5,3,2)$, \\
        & $(7,6,5,4,2)$, $(7,6,5,4,3)$, $(8,3,2,2,1)$, $(8,4,2,1,1)$, $(8,4,3,2,1)$, $(8,4,3,3,2)$, $(8,5,2,2,1)$, \\
        & $(8,5,4,2,1)$, $(8,5,4,3,2)$, $(8,6,3,2,1)$, $(8,6,4,1,1)$, $(8,6,4,3,1)$, $(8,6,5,2,1)$, $(8,6,5,4,1)$, \\
        & $(8,6,5,4,3)$, $(8,7,3,2,2)$, $(8,7,4,2,1)$, $(8,7,4,3,2)$, $(8,7,5,4,2)$, $(8,7,6,3,2)$, $(8,7,6,5,4)$, \\
        & $(9,4,3,2,1)$, $(9,5,3,2,1)$, $(9,6,4,2,1)$, $(9,6,5,2,1)$, $(9,7,4,3,1)$, $(9,7,5,3,1)$, $(9,8,4,3,2)$, \\
        & $(9,8,6,5,2)$, $(10,4,3,2,1)$, $(10,6,3,2,1)$, $(10,7,4,2,1)$, $(10,7,6,2,1)$, $(10,8,4,3,1)$, \\
        & $(10,8,6,3,1)$, $(10,9,4,3,2)$ \\\bottomrule
    \end{tabular}
    \caption{A list of all balanced vectors $\mathbf{v} \in \Z^p_{+}$ for $p \in \{1,2,3,4,5\}$, up to scaling and re-ordering of entries. There are $1$, $1$, $2$, $8$, and $113$ vectors, respectively, in the rows of this table.}\label{tab:totally_balanced_smallp}
\end{table}

\subsection{Regular vectors}\label{sec:regular_vec}

It seems likely that determining whether or not a given vector is balanced is a computationally difficult question; we are not aware of a significantly better method than just checking the null space of every $(p-1) \times p$ matrix with entries from $\{-1,0,1\}$. However, we will show shortly that the following vectors (whose defining condition is easy to check) are all balanced:

\begin{definition}\label{defn:regular_sequence}
    Suppose $\mathbf{v} = (v_1, v_2, \ldots, v_p) \in \Z^p_{+}$ is non-increasing (i.e., $v_1 \geq v_2 \geq \cdots \geq v_p$). If $v_j \leq \sum_{i=j+1}^p v_i$ for all $j \in \{1,2,\ldots,p-1\}$ and $v_p = 1$, then $\mathbf{v}$ is called \emph{regular}.
\end{definition}

We note that, in addition to the requirement $v_p = 1$ in this definition, the inequality $v_j \leq \sum_{i=j+1}^p v_i$ when $j = p-1$ forces $v_{p-1} = 1$, so all regular vectors end with $(\ldots, 1, 1)$. We also note that the adjective ``regular'' in the past was defined to refer to the sequence $v_1, v_2, \ldots, v_p$ \cite{FO89,oeisA003513}. Here, we instead apply it to the corresponding vector $(v_1, v_2, \ldots, v_p)$. That is, we talk about regular \emph{vectors} instead of regular \emph{sequences}; this is just a matter of convenience, as it is straightforward to convert between the two. Finally, we note that if $(v_1, v_2, \ldots, v_p)$ is a regular vector then $(v_1, v_2, \ldots, v_{p-1})$ is sometimes called a \emph{complete partition} (see \cite{Par98,oeisA126796} and the references therein).

The following proposition shows that regular vectors are balanced:

\begin{proposition}\label{prop:regular_sequences}
    Suppose $\mathbf{v} \in \Z^p_{+}$ is non-increasing and has $v_p = 1$. Then $\mathbf{v}$ is regular if and only if it is balanced and the matrix $A$ from Definition~\ref{defn:totally_balanced} can be chosen to be upper triangular.
\end{proposition}

\begin{proof}
    For the ``if'' direction, suppose $A \in \mathcal{M}_{p-1,p}$ is upper triangular, all of its entries belong to $\{-1,0,1\}$, and $\mathrm{null}(A) = \mathrm{span}(\mathbf{v})$. Furthermore, suppose without loss of generality that its diagonal entries all equal $1$ (if any equal $-1$ then we can multiply that row by $-1$ without changing the null space, and none can equal $0$ since that (together with the fact that $v_p = 1$) would imply $\mathrm{rank}(A) < p-1$).

    If $\mathbf{a_j}^T$ is the $j$-th row of $A$, then $\mathbf{v} \in \mathrm{null}(A)$ implies $\mathbf{a_j} \cdot \mathbf{v} = 0$, so (since $\mathbf{a_j} \in \{-1,0,1\}^p$) we have $v_j \leq \sum_{i=j+1}^p v_i$ for all $1 \leq j \leq p-1$. In other words, $\mathbf{v}$ is regular.
    
    Conversely, if $\mathbf{v}$ is regular then a standard result concerning regular vectors (see \cite{FO89} or \cite[Lemma~2.1]{FRM90}, for example) says that there exist subsets $S_j \subseteq \{j+1, j+2, \ldots, p\}$ such that $v_j = \sum_{i \in S_j} v_i$ for all $1 \leq j \leq p-1$. We then define $A \in \mathcal{M}_{p-1,p}$ to have its $(j,i)$-entry equal to
    \[
        a_{j,i} = \begin{cases}
            1, & \quad \text{if $i = j$,} \\
            -1, & \quad \text{if $i \in S_j$, and} \\
            0, & \quad \text{otherwise.}
        \end{cases}
    \]
    It is then straightforward to verify that $A$ is upper triangular, $A\mathbf{v} = \mathbf{0}$, and $\mathrm{rank}(A) = p-1$ (since its diagonal entries are all non-zero) so $\mathrm{null}(A) = \mathrm{span}(\mathbf{v})$. This completes the proof.
\end{proof}

For any dimension $p \geq 4$, regular vectors form a strict subset of balanced vectors. Indeed, it is straightforward to show that if a non-increasing vector $\mathbf{v} \in \Z^p_{+}$ is balanced then $v_1 \leq \sum_{i=2}^p v_i$ (just apply the triangle inequality to the first entry of $A\mathbf{v} = \mathbf{0}$, where $A$ is the matrix from Definition~\ref{defn:totally_balanced}). However, the other $p-2$ inequalities $v_j \leq \sum_{i=j+1}^p v_i$ for $j \in \{2,3,\ldots,p-1\}$ that hold for regular vectors need not hold for balanced vectors. For example, when $p = 4$ there are only $6$ regular vectors: $(1,1,1,1)$, $(2,1,1,1)$, $(3,1,1,1)$, $(2,2,1,1)$, $(3,2,1,1)$, and $(4,2,1,1)$. The other two (up to scaling) non-increasing balanced vectors in $\mathbb{Z}_{+}^4$ (i.e., $(3,2,2,1)$ and $(4,3,2,1)$) are not regular. When $p = 5$, only $27$ of the $113$ balanced vectors from Table~\ref{tab:totally_balanced_smallp} are regular.

\section{$\{-1,0,1\}$ and $\{-1,1\}$-diagonalizability of specific graph families}\label{sec:wHdiagGraphs}

In this section, we characterize $\{-1,0,1\}$- and $\{-1,1\}$-diagonalizability of complete graphs, complete multipartite graphs, graph complements, and some graphs that can be constructed by joining together smaller graphs.

The following lemma will be useful to us repeatedly when characterizing $\{-1,1\}$-diagonalizability of graphs.

\begin{lemma}\label{lem:one_neg_circulant}
    There exists an invertible matrix $A \in \mathcal{M}_n$ with the following properties if and only if $n \in \{1\} \cup \{2k : k \in \mathbb{Z}_{+}\}$:
    \begin{itemize}
        \item every entry of $A$ belongs to $\{-1,1\}$,
        
        \item one column of $A$ is $\mathbf{1}$, and

        \item all other columns of $A$ are orthogonal to $\mathbf{1}$.
    \end{itemize}
\end{lemma}

\begin{proof}
    The ``only if'' direction follows from the fact that if $n > 1$ is odd then there does not exist a vector (i.e., a column of $A$) in $\{-1,1\}^n$ that is orthogonal to $\mathbf{1}$, since the sum of an odd number members of $\{-1,1\}$ cannot equal $0$.
    
    For the ``if'' direction, the $n = 1$ case is clear and the $n \in \{2,4\}$ cases are handled by Hadamard matrices, so suppose $n \geq 6$ is even. Consider the vector $\mathbf{v}=(1,-1,1,-1,\ldots,1,-1) \in \R^n$. It is clear that the circulant matrix with top row equal to $\mathbf{v}$ has rank $1$ (every row in this matrix is the negative of the row above it), so \cite[Proposition 1.1]{ingleton} tells us that the polynomial
    \[
        f(x) := 1 - x + x^2 - x^3 + \cdots + x^{n-2} - x^{n-1}
    \]
    has $n-1$ of the $n$ different $n$-th roots of unity as roots. It is straightforward to check that $-1$ is not a root of $f$, so we conclude that all $n$-th roots of unity except for $-1$ are roots of $f$.

    Now consider the following polynomial, which is obtained from $f$ simply by swapping the signs of its first two coefficients (but leaving the rest unchanged):
    \[
        g(x) := (-1 + x) + (x^2 - x^3 + \cdots + x^{n-2} - x^{n-1}).
    \]
    Since $n \geq 6$ it is straightforward to see that $-1$ is not a root of $g$. Furthermore, since $(f-g)(x) = 2-2x$, has $1$ as its only root, we conclude that $1$ is the only root of unity that is a root of $g$. It follows from \cite[Proposition 1.1]{ingleton} that the circulant matrix $C$ with top row equal to $\mathbf{w}=(-1,1, 1, -1, \ldots, 1,-1)\in \R^n$ (obtained from $\mathbf{v}$ by swapping the signs of its first two entries) has rank $n-1$. We can thus remove one column from $C$ without changing its rank, and then append $\mathbf{1}$ as a new column in its place. Since $\mathbf{1}$ is orthogonal to all of the other columns, the resulting matrix $A$ will have full rank and all other properties described in the statement of the lemma.
\end{proof}

Our usage of the above lemma will be to demonstrate the existence or non-existence of a $\{-1,1\}$ matrix diagonalizing certain Laplacians. Since every Laplacian is symmetric with $\mathbf{1}$ as an eigenvector (corresponding to the eigenvalue $0$), we can always (for every graph Laplacian and every $n$) find a matrix $A$ diagonalizing a graph Laplacian that satisfies the second and third conditions of Lemma~\ref{lem:one_neg_circulant}; it is just the first condition that is trickier.

\begin{corollary}\label{cor:is_one_neg}
    Suppose $n \in \mathbb{Z}_{+}$. The following are equivalent:
    \begin{itemize}
        \item[a)] $n \in \{1\} \cup \{2k : k \in \mathbb{Z}_{+}\}$;

        \item[b)] there is an $n$-vertex connected $\{-1,1\}$-diagonalizable graph; and

        \item[c)] the complete graph $K_n$ is $\{-1,1\}$-diagonalizable.
    \end{itemize}
\end{corollary}

\begin{proof}
    The fact that (c) implies (b) is trivial.
    
    To see that (b) implies (a), recall that in a connected graph $G$ all eigenvectors of the Laplacian $L$ except for the scalar multiples of $\mathbf{1}$ (i.e., all eigenvectors corresponding to an eigenvalue other than $0$) are orthogonal to $\mathbf{1}$. If $L$ is diagonalized by a matrix with entries from $\{-1,1\}$, the columns of that matrix are eigenvectors of $L$ and is thus a matrix with all of the properties described by Lemma~\ref{lem:one_neg_circulant}. By that lemma, $n \in \{1\} \cup \{2k : k \in \mathbb{Z}_{+}\}$.
    
    Finally, to see that (a) implies (c), let $A \in \mathcal{M}_n$ be a matrix with all of the properties described by Lemma~\ref{lem:one_neg_circulant}. The Laplacian matrix $L$ of $K_n$ is $L = nI - J$, which has eigenspaces $\mathrm{span}(\mathbf{1})$ (corresponding to eigenvalue $0$) and its orthogonal complement (corresponding to eigenvalue $n$). It follows that $A$ (whose columns all equal $\mathbf{1}$ or are orthogonal to $\mathbf{1}$) diagonalizes $L$, so $K_n$ is $\{-1,1\}$-diagonalizable. 
\end{proof}

\subsection{The complete graph}

The following theorem summarizes what is known about the $\{-1,0,1\}$-diagonalizability and $\{-1,1\}$-diagonalizability and bandwidths of complete graphs.

\begin{theorem}\label{thm:complete_is_pm1}
    Let $n \geq 1$ be an integer and let $K_n$ be the complete graph on $n$ vertices.
    \begin{itemize}
        \item[a)] $K_n$ is $\{-1,0,1\}$-diagonalizable, with $\{-1,0,1\}$-bandwidth at most $2$.
        
        \item[b)] $K_n$ is $\{-1,1\}$-diagonalizable if and only if $n$ is even. Furthermore,

        \begin{itemize}
            \item[i)] if $n \equiv 2 \pmod{4}$ then $K_n$ has $\{-1,1\}$-bandwidth $n-1$; and

            \item[ii)] if there exists a Hadamard matrix of order $n$ (and thus $n \in \{1,2\} \cup \{4k : k \in \mathbb{Z}_{+}\}$) then $K_n$ has $\{-1,1\}$-bandwidth $1$.
        \end{itemize}
    \end{itemize}
\end{theorem}

\begin{proof}
    Part~(a) of the theorem is equivalent to the statement that $K_n$ is weakly Hadamard diagonalizable, which was proved in \cite[Lemma~1.5]{adm2021weakly}. The non-bandwidth portion of part~(b) was proved already in Corollary~\ref{cor:is_one_neg}. The bandwidth statement (b)(ii) is equivalent to the statement that if there exists a Hadamard matrix of order $n$ then $K_n$ is Hadamard-diagonalizable, which was proved in \cite[Proposition~2.3]{Breen22}. Thus all that remains is to prove the bandwidth claim (b)(i).
    
    To verify (b)(i), suppose that $n \equiv 2 \pmod{4}$. Since the Laplacian $L$ of $K_n$ has eigenspaces of dimension $1$ and $n-1$ (corresponding to eigenvalues $0$ and $n$, respectively), it is immediate that its $\{-1,1\}$-bandwidth is no larger than $n-1$: if $P$ diagonalizes $L$ then $P^TP$ will be block diagonal with blocks of size $1 \times 1$ and $(n-1) \times (n-1)$. To see that the $\{-1,1\}$-bandwidth is not smaller than $n-1$, it suffices to show that none of its $\{-1,1\}$-eigenvectors corresponding to eigenvalue $n$ are orthogonal to each other (so all entries in the $(n-1) \times (n-1)$ diagonal block of $P^TP$ are non-zero). To this end, simply notice that if $\mathbf{v}$ and $\mathbf{w}$ are $\{-1,1\}$-eigenvectors of $L$ corresponding to eigenvalue $n$ then they are both orthogonal to $\mathbf{1}$ and thus have $n/2$ entries equal to each of $-1$ and $1$. If $k$ denotes the number of subscripts $j$ for which $v_j = w_j = 1$ then we see that
    \[
        \mathbf{v} \cdot \mathbf{w} = 2k - 2((n/2) - k) = 4k - n \neq 0
    \]
    (since $n \equiv 2 \pmod{4}$), so $\mathbf{v}$ and $\mathbf{w}$ are not orthogonal.
\end{proof}

The above theorem leaves open some questions about $\{-1,1\}$-bandwidth and $\{-1,0,1\}$-bandwidth of the complete graph in certain cases. In particular, when are the $\{-1,1\}$- and $\{-1,0,1\}$-bandwidths of $K_n$ equal to $1$ and when are they equal to $2$? We now show that pinning down these remaining bandwidth questions is likely very difficult, as they are equivalent to known and presumably difficult combinatorial questions:

\begin{theorem}\label{thm:hadamard_equiv}
    Suppose $n \in \mathbb{Z}_{+}$. The following are equivalent:
    \begin{itemize}
        \item[a)] there exists an $n \times n$ Hadamard matrix;

        \item[b)] there is an $n$-vertex connected graph with $\{-1,1\}$-bandwidth $1$; and
        
        \item[c)] the complete graph $K_n$ has $\{-1,1\}$-bandwidth $1$.
    \end{itemize}
\end{theorem}

\begin{proof}
    The fact that (c) implies (b) is trivial.

    To see that (b) implies (a), suppose $G$ is an $n$-vertex connected graph with $\{-1,1\}$-bandwidth $1$, $L$ is its Laplacian matrix, and $P$ is an invertible matrix whose entries belong to $\{-1,1\}$ for which $P^{-1}LP$ is diagonal and $P^TP$ is diagonal (the existence of such a $P$ follows from the fact that $G$ has $\{-1,1\}$-bandwidth $1$). This implies exactly that $P$ is a Hadamard matrix.

    Finally, to see that (a) implies (c), recall from \cite[Proposition~2.3]{Breen22} that every Hadamard matrix diagonalizes the Laplacian of $K_n$, so it has $\{-1,1\}$-bandwidth $1$ whenever a Hadamard matrix of order $n$ exists.
\end{proof}

A well-known and long-standing (and thus likely very difficult) conjecture aims to establish when the equivalent conditions of Theorem~\ref{thm:hadamard_equiv} hold:\footnote{The ``only if'' direction of this conjecture is known to hold; the conjecture really concerns the ``if'' direction.}

\begin{conjecture}[Hadamard conjecture]\label{conj:hadamard}
    The (equivalent) statements of Theorem~\ref{thm:hadamard_equiv} hold if and only if $n \in \{1,2\} \cup \{4k : k \in \mathbb{Z}_{+}\}$.
\end{conjecture}

The exact value of the $\{-1,0,1\}$-bandwidth of $K_n$ is equivalent to the existence of a Hadamard-like matrix that is allowed to have some entries equal to $0$:

\begin{theorem}\label{thm:hadamard_equiv_with_zero}
    Suppose $n \in \mathbb{Z}_{+}$. The following are equivalent:
    \begin{itemize}
        \item[a)] there exists an $n \times n$ matrix with entries from $\{-1,0,1\}$, with mutually orthogonal columns and one column equal to $\mathbf{1}$;

        \item[b)] there is an $n$-vertex connected graph with $\{-1,0,1\}$-bandwidth $1$; and
        
        \item[c)] the complete graph $K_n$ has $\{-1,0,1\}$-bandwidth $1$.
    \end{itemize}
\end{theorem}

\begin{proof}
    The fact that (c) implies (b) is trivial.

    To see that (b) implies (a), suppose $G$ is an $n$-vertex connected graph with $\{-1,0,1\}$-bandwidth $1$, $L$ is its Laplacian matrix, and $P$ is an invertible matrix whose entries belong to $\{-1,0,1\}$ for which $P^{-1}LP$ is diagonal and $P^TP$ is diagonal (the existence of such a $P$ follows from the fact that $G$ has $\{-1,0,1\}$-bandwidth $1$). This implies exactly that $P$ has the properties described by part~(a) of the theorem.

    Finally, to see that (a) implies (c), let $P$ be a matrix with the properties described by part~(a) of the theorem and let $L$ be the Laplacian matrix of $K_n$. Since the eigenspaces of $L$ are $\mathrm{span}(\mathbf{1})$ and its orthogonal complement, $P$ diagonalizes $L$: $P^{-1}LP$ is diagonal. Furthermore, mutual orthogonality of the columns of $P$ tells us that $P^TP$ is diagonal, so $L$ has $\{-1,0,1\}$-bandwidth $1$.
\end{proof}

Given a positive integer $n$, determining whether or not there exists an $n \times n$ matrix with the properties described by part~(a) of Theorem~\ref{thm:hadamard_equiv_with_zero} seems to be quite difficult, so we state it as another conjecture:\footnote{Both the ``if'' and ``only if'' directions of this conjecture are unsolved.}

\begin{conjecture}\label{conj:hadamard_zero}
    The (equivalent) statements of Theorem~\ref{thm:hadamard_equiv_with_zero} hold if and only if $n \in \{1,2\} \cup \{4k : k \in \mathbb{Z}_{+}\}$.
\end{conjecture}

We posed Conjecture~\ref{conj:hadamard_zero} on MathOverflow \cite{over}, and via a combined effort of Max Alekseyev, Ilya Bogdanov, and Mikhail Tikhomirov, it has been proved in the following special cases:
\begin{itemize}
    \item when $n \leq 20$;

    \item when $n = p^k$ for some $k \in \mathbb{Z}_{+}$ and some odd prime $p$; and

    \item when $n = 2p^k$ for some $k \in \mathbb{Z}_{+}$ and some odd prime $p$.
\end{itemize}
It follows that $K_n$ has $\{-1,0,1\}$-bandwidth equal to $1$ when $n \in \{1,2,4,8,12,16,20\}$ (or more generally, whenever an $n \times n$ Hadamard matrix exists) and it has $\{-1,0,1\}$-bandwidth equal to $2$ when $n \in \{3,5,6,7,9,10,11,13,14,15,17,18,19\}$, when $n$ is an odd prime power, and when $n$ is double an odd prime power. The general case remains open.

\subsection{Graph complements and joins}\label{sec:complements}

We noted in Proposition~\ref{prop:neg101properties}(a) that the complement of a connected $\{-1,0,1\}$-diagonalizable graph is also $\{-1,0,1\}$-diagonalizable. The situation for complements of disconnected graphs is much more complicated:

\begin{theorem}\label{thm:complement_of_neg_zero_one}
    Suppose that $G$ is a $\{-1,0,1\}$-diagonalizable graph with $p$ connected components of sizes $v_1$, $v_2$, $\ldots$, $v_p$, respectively. Then $G^{\textup{c}}$ is $\{-1,0,1\}$-diagonalizable if and only if the vector $(v_1,v_2,\ldots,v_p)$ is balanced. Furthermore, if $G$ has $\{-1,0,1\}$-bandwidth $k$ then $G^{\textup{c}}$ has $\{-1,0,1\}$-bandwidth at most $\max\{p-1,k\}$.
\end{theorem}

The above theorem provides a significant strengthening and generalization of several known results about $\{-1,0,1\}$-diagonalizability of graphs (however, the bandwidth bound $\max\{p-1,k\}$ is typically weaker than can be obtained by investigating specific instances):

\begin{itemize}
    \item Applying Theorem~\ref{thm:complement_of_neg_zero_one} to the graph $G = K_1 \sqcup K_1 \sqcup \cdots \sqcup K_1$ (which is trivially $\{-1,0,1\}$-diagonalizable) recovers the fact that, since the vector $(1,1,\ldots,1)$ is balanced, $G^\textup{c} = K_p$ is also $\{-1,0,1\}$-diagonalizable (refer back to Theorem~\ref{thm:complete_is_pm1}(a)).

    \item Applying Theorem~\ref{thm:complement_of_neg_zero_one} to the graph $G = K_2 \sqcup \cdots \sqcup K_2 \sqcup K_1 \sqcup \cdots \sqcup K_1$ (which is $\{-1,0,1\}$-diagonalizable by Proposition~\ref{prop:neg101properties}(b)) tells us (since the vector $(2\mathbf{1}_p,\mathbf{1}_q)$ is regular and thus balanced whenever $q \geq 2$) that $G^{\textup{c}}$ is $\{-1,0,1\}$-diagonalizable. In fact, one can easily verify that $G^\textup{c}$ is the complete graph $K_{2p+q}$ with $p$ independent edges removed, so we recover the results \cite[Lemma~4.8]{adm2021weakly} and \cite[Corollary~4.9]{adm2021weakly}, but with a weaker bandwidth bound.

    \item When $G = K_k \sqcup K_n^\textup{c}$ then (since $K_k$ and $K_n^\textup{c} = K_1 \sqcup \cdots \sqcup K_1$ are $\{-1,0,1\}$-diagonalizable) we learn from Theorem~\ref{thm:complement_of_neg_zero_one} that $G^\textup{c} = K_k^\textup{c} \vee K_n$ is $\{-1,0,1\}$-diagonalizable whenever $k \geq n$ (compare with \cite[Lemma~4.6]{adm2021weakly}, which was a similar result under the stronger hypothesis that $n-k \in \{0,1,2\}$).
    
    \item If $G$ is $\{-1,0,1\}$-diagonalizable and has all of its connected components of the same size, Theorem~\ref{thm:complement_of_neg_zero_one} tells us that $G^\textup{c}$ is also $\{-1,0,1\}$-diagonalizable (since the vector $(1,1,\ldots,1)$ is regular and thus balanced). Compare this with \cite[Lemma~2.5]{adm2021weakly}, where a similar result was established under the additional hypothesis that $G$ is regular.\footnote{However, the proof of \cite[Lemma~2.5]{adm2021weakly} does not make use of regularity at all; it could be removed from the list of hypotheses.}
\end{itemize}

\begin{proof}[Proof of Theorem~\ref{thm:complement_of_neg_zero_one}]
    We assume that $p > 1$ (i.e., $G$ is disconnected) throughout this proof, since the $p = 1$ case is covered by Proposition~\ref{prop:neg101properties}(a).
    
    If we denote the Laplacian matrix of $G$ by $L$ then the Laplacian matrix of $G^{\textup{c}}$ is $L^{\textup{c}} := nI - L - J$. It is straightforward to show that every eigenvector of $L$ corresponding to a non-zero eigenvalue is also an eigenvector of $L^{\textup{c}}$. We would like to make a similar claim about eigenvectors corresponding to the eigenvalue $0$ of $L$, but we have to be considerably more careful in this case.

    If we assume that the vertices of $G$ are arranged so that its first $v_1$ vertices belong to the first connected component, the next $v_2$ vertices belong to its second connected component, and so on, the $p$-dimensional eigenspace of $L$ corresponding to the eigenvalue $0$ is
    \begin{align}\label{eq:comp_zero_evecs}
        \big\{ (c_1\mathbf{1}_{v_1}, c_2\mathbf{1}_{v_2}, \ldots, c_p\mathbf{1}_{v_p}) : c_1,c_2,\ldots,c_p \in \R \big\}.
    \end{align}

    Since the complement of a disconnected graph is always connected, we know that $G^{\textup{c}}$ is connected and thus the eigenspace of $L^{\textup{c}}$ corresponding to eigenvalue $0$ is $1$-dimensional: it is the subspace of the set~\eqref{eq:comp_zero_evecs} with $c_1 = c_2 = \cdots = c_p$. Thus any other eigenvector $(c_1\mathbf{1}_{v_1}, c_2\mathbf{1}_{v_2}, \ldots, c_p\mathbf{1}_{v_p})$ of $L^{\textup{c}}$ in the set~\eqref{eq:comp_zero_evecs} must be orthogonal to $(\mathbf{1}_{v_1}, \mathbf{1}_{v_2}, \ldots, \mathbf{1}_{v_p})$, so must have $\sum_{j=1}^p c_jv_j = 0$. Conversely, any vector $\mathbf{v}$ of this form with $\sum_{j=1}^p c_jv_j = 0$ is indeed an eigenvector of $L^{\textup{c}}$ since
    \[
        (nI - L - J)\mathbf{v} = n\mathbf{v} - 0\mathbf{v} - \left(\sum_{j=1}^p c_jv_j\right)\mathbf{1} = n\mathbf{v}.
    \]

    It follows that we can find a basis of eigenvectors of $L^{\textup{c}}$ with entries from $\{-1,0,1\}$ if and only if the following conditions hold:
    \begin{itemize}
        \item[a)] each non-zero eigenspace of $L$ has a basis of eigenvectors with entries from $\{-1,0,1\}$; and

        \item[b)] there is a basis of the $(p-1)$-dimensional subspace
        \begin{align}\label{eq:comp_Gc_zero_evecs}
            \left\{ (c_1\mathbf{1}_{v_1}, c_2\mathbf{1}_{v_2}, \ldots, c_p\mathbf{1}_{v_p}) : \sum_{j=1}^p c_jv_j = 0 \right\}
        \end{align}
        consisting of vectors with entries from $\{-1,0,1\}$.
    \end{itemize}

    Property~(a) above holds by assumption: one of our hypotheses is that $G$ is $\{-1,0,1\}$-diagonalizable. We claim that property~(b) holds if and only if $(v_1,v_2,\ldots,v_p)$ is balanced. In fact, this follows almost immediately from Definition~\ref{defn:totally_balanced}: such a basis exists if and only if there is a linearly independent set $\{\mathbf{a}_1,\mathbf{a}_2,\ldots,\mathbf{a}_{p-1}\} \subset \{-1,0,1\}^p$ with $\mathbf{a}_j \cdot \mathbf{v} = 0$ for all $1 \leq j \leq p-1$ (each $\mathbf{a}_j$ vector is of the form $(c_1,c_2,\ldots,c_p)$ for some choice of $c_1,c_2,\ldots,c_p$ from Equation~\eqref{eq:comp_Gc_zero_evecs}). The matrix $A$ from Definition~\ref{defn:totally_balanced} is the one whose rows are $\mathbf{a}_1$, $\mathbf{a}_2$, $\ldots$, $\mathbf{a}_{p-1}$, which completes the proof that $G^{\textup{c}}$ is $\{-1,0,1\}$-diagonalizable if and only if $(v_1,v_2,\ldots,v_p)$ is balanced.

    The claim about the $\{-1,0,1\}$-bandwidth of $G^{\textup{c}}$ follows just from investigating the orthogonality relationships between the eigenvectors in the $\{-1,0,1\}$-diagonalization of $L^{\textup{c}}$ that we constructed. The eigenvector $\mathbf{1}$ is orthogonal to all other eigenvectors, the $p-1$ eigenvectors from the set~\eqref{eq:comp_Gc_zero_evecs} are orthogonal to all other eigenvectors (but not necessarily to each other), and the remaining eigenvectors were all eigenvectors of $L$ itself. Since $G$ has $\{-1,0,1\}$-bandwidth $k$, this gives an upper bound of $\max\{1,p-1,k\} = \max\{p-1,k\}$ on the $\{-1,0,1\}$-bandwidth of $G^{\textup{c}}$. 
\end{proof}

The above result gives us a reasonably complete characterization of when the join of $\{-1,0,1\}$-diagonalizable graphs is also $\{-1,0,1\}$-diagonalizable:

\begin{corollary}\label{cor:join_neg_zero_one}
    Let $p \in \mathbb{Z}_{+}$. Suppose that, for $1 \leq j \leq p$, $G_j$ is a connected $v_j$-vertex $\{-1,0,1\}$-diagonalizable graph with $\{-1,0,1\}$-bandwidth $k_j$. Then the join
    \[
        G_1^{\textup{c}} \vee G_2^{\textup{c}} \vee \cdots \vee G_p^{\textup{c}}
    \]
    is $\{-1,0,1\}$-diagonalizable if and only if the vector $(v_1, v_2, \ldots, v_p)$ is balanced. Furthermore, it has $\{-1,0,1\}$-bandwidth at most $\max\{p-1,k_1,k_2,\ldots,k_p\}$.
\end{corollary}

\begin{proof}
    Recall that the join of graphs can be written as the complement of the disjoint union of their complements:
    \begin{align}\label{eq:join_via_complement}
        G_1^{\textup{c}} \vee G_2^{\textup{c}} \vee \cdots \vee G_p^{\textup{c}} = (G_1 \sqcup G_2 \sqcup \cdots \sqcup G_p)^{\textup{c}}.
    \end{align}
    
    Proposition~\ref{prop:neg101properties}(b) tells us that $G_1 \sqcup G_2 \sqcup \cdots \sqcup G_p$ is $\{-1,0,1\}$-diagonalizable with $\{-1,0,1\}$-bandwidth at most $\max\{k_1,k_2,\ldots,k_p\}$. Now we just use the fact that $G_1 \sqcup G_2 \sqcup \cdots \sqcup G_p$ is a graph with $p$ connected components that have $v_1$, $v_2$, $\ldots$, $v_p$ vertices, respectively, so Theorem~\ref{thm:complement_of_neg_zero_one} tells us that its complement~\eqref{eq:join_via_complement} is $\{-1,0,1\}$-diagonalizable if and only if $(v_1, v_2, \ldots, v_p)$ is balanced, and its $\{-1,0,1\}$-bandwidth at most $\max\{p-1,k_1,k_2,\ldots,k_p\}$.
\end{proof}

\subsection{Complete multipartite graphs}\label{sec:complete_multipartite}

If we specialize Corollary~\ref{cor:join_neg_zero_one} even further, we learn exactly which complete multipartite graphs are $\{-1,0,1\}$-diagonalizable:

\begin{corollary}\label{cor:complete_multipartite_graph}
    The complete multipartite graph $K_{v_1,v_2,\ldots,v_p}$ is $\{-1,0,1\}$-diagonalizable if and only if the vector $(v_1,v_2,\ldots,v_p)$ is balanced. If it is $\{-1,0,1\}$-diagonalizable then its $\{-1,0,1\}$-bandwidth is at most $\max\{2,p-1\}$.
\end{corollary}

\begin{proof}
    Recall that the complete multipartite graph is the join of complements of complete graphs:
    \begin{align}\label{eq:comp_multi_via_join_of_complete}
        K_{v_1,v_2,\ldots,v_p} = K_{v_1}^\textup{c} \vee K_{v_2}^\textup{c} \vee \cdots \vee K_{v_p}^\textup{c}.
    \end{align}
    Theorem~\ref{thm:complete_is_pm1}(a) tells us that every complete graph is $\{-1,0,1\}$-diagonalizable with $\{-1,0,1\}$-bandwidth at most $2$. Corollary~\ref{cor:join_neg_zero_one} then tells us that the join~\eqref{eq:comp_multi_via_join_of_complete} is $\{-1,0,1\}$-diagonalizable if and only if $(v_1,v_2,\ldots,v_p)$ is balanced, and its $\{-1,0,1\}$-bandwidth is at most $\max\{2,p-1\}$.
\end{proof}

Corollary~\ref{cor:complete_multipartite_graph} can be thought of as a generalization of \cite[Lemma~2.6]{adm2021weakly}, which said that a complete bipartite graph $K_{m,n}$ is $\{-1,0,1\}$-diagonalizable if and only if $m = n$ (this lemma was stated in terms of diagonalizability by weakly Hadamard matrices, but the proof applies just as well to $\{-1,0,1\}$-diagonalizability). In particular, if $p = 2$ then the only balanced vectors in $\mathbb{Z}_{+}^p$ are the multiples of $(1,1)$ (refer to Table~\ref{tab:totally_balanced_smallp}), so Corollary~\ref{cor:complete_multipartite_graph} says that $K_{m,n}$ is $\{-1,0,1\}$-diagonalizable if and only if $m = n$.

Similarly, in the $p = 3$ case, the only balanced vectors in $\mathbb{Z}_{+}^p$ are the multiples and permutations of $(1,1,1)$ and $(2,1,1)$ (refer back to Table~\ref{tab:totally_balanced_smallp}), so the $p = 3$ case of Corollary~\ref{cor:complete_multipartite_graph} tells us that $K_{\ell,m,n}$ is $\{-1,0,1\}$-diagonalizable if and only if either $\ell = m = n$ or (assuming $\ell \geq m \geq n$) $\ell = 2m = 2n$.

We now give a similar characterization of $\{-1,1\}$-diagonlizability of complete multipartite graphs, which is even more explicit (and much more restrictive):

\begin{theorem}\label{thm:complete_multipartite_graph_oneneg}
    The complete multipartite graph $K_{v_1,v_2,\ldots,v_p}$ is $\{-1,1\}$-diagonalizable if and only if $v_1 = v_2 = \cdots = v_p$, and $p, v_1 \in \{1\} \cup \{2k : k \in \mathbb{Z}_{+}\}$.
\end{theorem}

\begin{proof}
    If we define $\widetilde{v_j} := \sum_{i \neq j} v_i$ then the Laplacian matrix of $K_{v_1,v_2,\ldots,v_p}$ is
    \begin{align*}
        L = \begin{bmatrix}
            \widetilde{v_1}I_{v_1} & -J_{v_1,v_2} & -J_{v_1,v_3} & \cdots & -J_{v_1,v_p} \\
            -J_{v_2,v_1} & \widetilde{v_2}I_{v_2} & -J_{v_2,v_3} & \cdots & -J_{v_2,v_p} \\
            -J_{v_3,v_1} & -J_{v_3,v_2} & \widetilde{v_3}I_{v_3} & \cdots & -J_{v_3,v_p} \\
            \vdots & \vdots & \vdots & \ddots & \vdots \\
            -J_{v_p,v_1} & -J_{v_p,v_2} & -J_{v_p,v_3} & \cdots & \widetilde{v_p}I_{v_p}
        \end{bmatrix}.
    \end{align*}
    It is straightforward to verify that, for all $1 \leq i \leq p$ and $1 \leq j \leq v_i-1$, the vector
    \begin{align}\label{eq:cmg_first_evecs}
        (\mathbf{0}_{v_1}, \ldots, \mathbf{0}_{v_{i-1}}, \mathbf{e}_j - \mathbf{e}_{j+1}, \mathbf{0}_{v_{i+1}}, \ldots, \mathbf{0}_{v_p})
    \end{align}
    is an eigenvector of $L$ with corresponding eigenvalue equal to $\widetilde{v_i}$. Together with $\mathbf{1}$ (i.e., the eigenvector of $L$ with corresponding eigenvalue equal to $0$), this gives us a set of $(\sum_i v_i) - (p-1)$ linearly independent eigenvectors of $L$. The remaining $(p-1)$-dimensional eigenspace has corresponding eigenvalue $\sum_i v_i$ and is equal to
    \begin{align}\label{eq:cmg_last_evecs}
        \mathrm{span}\left\{ (c_1\mathbf{1}_{v_1}, c_2\mathbf{1}_{v_2}, \ldots, c_p\mathbf{1}_{v_p}) : \sum_i c_iv_i = 0 \right\}.
    \end{align}

    With the structure of the eigenspaces of $L$ understood, we now begin the ``only if'' direction of the proof. Suppose (for the sake of establishing a contradiction) that there exist some indices $i,j$ such that $v_i \neq v_j$ and thus $\widetilde{v_i} \neq \widetilde{v_j}$. Since the eigenspace of $L$ corresponding to the eigenvalue $\widetilde{v_i}$ is spanned by vectors of the form~\eqref{eq:cmg_first_evecs}, which all have $\mathbf{0}_{v_j}$ in their $j$-th block. It follows that there is no $\{-1,1\}$ vector in this eigenspace, so $K_{v_1,v_2,\ldots,v_p}$ is not $\{-1,1\}$-diagonalizable.

    Now suppose that $K_{v_1,v_2,\ldots,v_p}$ is $\{-1,1\}$-diagonalizable. To see that $v_1$ must be even or equal to $1$, recall (again, from Equation~\eqref{eq:cmg_first_evecs}) that if $v_1 > 1$ then every eigenvector corresponding to the eigenvalue $\widetilde{v_1}$ must have its first $v_1$ entries add to $0$. If such an eigenvector only has entries from $\{-1,1\}$ then $v_1$ must be even.

    To see that $p$ must be even or equal to $1$, recall that if $p > 1$ then the eigenspace corresponding to the eigenvalue $\sum_i v_i$ is as in Equation~\eqref{eq:cmg_last_evecs}. If a vector $(a_1\mathbf{1}_{v_1}, a_2\mathbf{1}_{v_2},\ldots,a_p\mathbf{1}_{v_p})$ with $\sum_{i=1}^p a_iv_i = v_1\sum_{i=1}^p a_i = 0$ has $a_i \in \{-1,1\}$ for all $1 \leq i \leq p$ then $p$ must be even. This complete the proof of the ``only if'' direction.

    Conversely, suppose that $v_1 = v_2 = \cdots = v_p$ and $p, v_1 \in \{1\} \cup \{2k : k \in \mathbb{Z}_{+}\}$. Then the Laplacian matrix $L$ of $K_{v_1,v_2,\ldots,v_p}$ has eigenvalues $0$, $\widetilde{v_1}$ and $pv_1$ with multiplicities $1$, $p(v_1-1)$ and $p-1$, respectively. The eigenspace corresponding to the eigenvalue $0$ is spanned by the vector $\mathbf{1}$. The eigenspace corresponding to the eigenvalue $\widetilde{v_1}$ is
    \begin{align}\label{eq:v1tilde_eigenspace}
        \big\{ (\mathbf{x_1}, \mathbf{x_2}, \ldots, \mathbf{x_p} ) : \mathbf{x_j} \cdot \mathbf{1} = 0 \ \text{for all} \ 1 \leq j \leq p \big\}.
    \end{align}
    If $v_1 = 1$ then this eigenspace is $0$-dimensional. If $v_1 > 1$ is even then, by Lemma~\ref{lem:one_neg_circulant}, there exists a full rank matrix $X \in \mathcal{M}_{v_1,v_1-1}$ with entries from $\{-1,1\}$ whose columns are all orthogonal to $\mathbf{1}$. If we let $\mathbf{x_j}$ be the $j$-th column of $X$ ($1 \leq j \leq v_1-1$), then the set
    \begin{align}\label{eq:spanning_set_1negzero}
        \big\{ (\mathbf{x_{j_1}}, \mathbf{x_{j_2}}, \ldots, \mathbf{x_{j_p}} ) : 1 \leq j_i \leq v_1-1 \ \text{for all} \ 1 \leq i \leq p \big\}
    \end{align}
    spans the eigenspace~\eqref{eq:v1tilde_eigenspace}, and each of its members has entries coming from $\{-1,1\}$. There is thus some subset of the set~\eqref{eq:spanning_set_1negzero} that is a $\{-1,1\}$-basis of the eigenspace~\eqref{eq:v1tilde_eigenspace}.

    Finally, the eigenspace corresponding to the eigenvalue $pv_1$ is
    \begin{align}\label{eq:pv1_eigenspace}
        \big\{ \mathbf{x} \otimes \mathbf{1}_{v_1} : \mathbf{x} \cdot \mathbf{1}_p = 0 \big\}.
    \end{align}
    If $p = 1$ then this eigenspace is $0$-dimensional. If $p > 1$ is even then, similarly to before, Lemma~\ref{lem:one_neg_circulant} tells us that there exists a full rank matrix $X \in \mathcal{M}_{p-1}$ with entries from $\{-1,1\}$ whose columns are all orthogonal to $\mathbf{1}$. If we let $\mathbf{x_j}$ be the $j$-th column of $X$ ($1 \leq j \leq p-1$), then the set
    \[
        \big\{ \mathbf{x_{j}} \otimes \mathbf{1}_{v_1} : 1 \leq j \leq p-1 \big\}
    \]
    is a $\{-1,1\}$-basis of the eigenspace~\eqref{eq:pv1_eigenspace}. It follows that every eigenspace has a basis consisting of vectors whose entries belong to $\{-1,1\}$, which completes the proof.
\end{proof}


Theorem~\ref{thm:complete_multipartite_graph_oneneg} shows that there are very few (fewer than $n$) complete multipartite graphs on $n$ vertices that are $\{-1,1\}$-diagonalizable. By contrast, our final result of this subsection gives a lower bound that shows that there are many $\{-1,0,1\}$-diagonalizable complete multipartite graphs on $n$ vertices. More specifically, there are superpolynomially many of them:

\begin{corollary}\label{cor:lots_of_diag_graphs}
    There exists $N \in \mathbb{Z}_{+}$ such that, whenever $n \geq N$, there are at least $13^{\sqrt{n}}$ complete multipartite graphs on $n$ vertices that are $\{-1,0,1\}$-diagonalizable but not $\{-1,1\}$-diagonalizable.
\end{corollary}

\begin{proof}
    If $f(n)$ denotes the number of complete partitions of $n$ then it is known (see \cite{oeisA126796} and the references therein) that
    \[
        \lim_{n\rightarrow\infty} \frac{4\sqrt{3} n f(n)}{e^{\pi\sqrt{2n/3}}} = 1.
    \]
    Since $e^{\pi\sqrt{2/3}} \approx 13.0019 > 13$, it follows that there exists $N \in \mathbb{Z}_{+}$ such that
    \[
        \frac{f(n-1) - n}{13^{\sqrt{n}}} > 1
    \]
    whenever $n \geq N$. Since $f(n-1)$ equals the number of regular vectors with sum equal to $n$, we conclude that $f(n-1) - n$ is a lower bound on the number of complete multipartite graphs on $n$ vertices that are $\{-1,0,1\}$-diagonalizable but not $\{-1,1\}$-diagonalizable. The result follows.
\end{proof}

For small values of $n$, we can compute the exact number of $n$-vertex $\{-1,0,1\}$-diagonalizable complete multipartite graphs by computing the number of balanced vectors in $\Z^p_{+}$ (for any $p$) with non-decreasing entries that sum to $n$. A list of all such vectors for $1 \leq n \leq 9$ is provided in Table~\ref{tab:totally_balanced_smallsum}.

\begin{table}[!htb]
    \centering
    \begin{tabular}{cl}\toprule
        $n$ & balanced vectors $\mathbf{v}$ \\\toprule
        $1$ & $(1)$ \\\midrule
        $2$ & $(1,1)$, $(2)$ \\\midrule
        $3$ & $(1,1,1)$, $(3)$ \\\midrule
        $4$ & $(1,1,1,1)$, $(2,1,1)$, $(2,2)$, $(4)$ \\\midrule
        $5$ & $(1,1,1,1,1)$, $(2,1,1,1)$, $(5)$ \\\midrule
        $6$ & $(1,1,1,1,1,1)$, $(2,1,1,1,1)$, $(2,2,1,1)$, $(3,1,1,1)$, $(2,2,2)$, $(3,3)$, $(6)$ \\\midrule
        $7$ & $(1,1,1,1,1,1,1)$, $(2,1,1,1,1,1)$, $(2,2,1,1,1)$, $(3,1,1,1,1)$, $(3,2,1,1)$, $(7)$ \\\midrule
        $8$ & $(1,1,1,1,1,1,1,1)$, $(2,1,1,1,1,1,1)$, $(2,2,1,1,1,1)$, $(3,1,1,1,1,1)$, $(2,2,2,1,1)$, \\
        & $(3,2,1,1,1)$, $(4,1,1,1,1)$, $(2,2,2,2)$, $(3,2,2,1)$, $(4,2,1,1)$, $(4,2,2)$, $(4,4)$, $(8)$ \\\midrule
        $9$ & $(1,1,1,1,1,1,1,1,1)$, $(2,1,1,1,1,1,1,1)$, $(2,2,1,1,1,1,1)$, $(3,1,1,1,1,1,1)$, $(2,2,2,1,1,1)$, \\
        & $(3,2,1,1,1,1)$, $(4,1,1,1,1,1)$, $(3,2,2,1,1)$, $(3,3,1,1,1)$, $(4,2,1,1,1)$, $(3,3,3)$, $(9)$ \\\bottomrule
    \end{tabular}
    \caption{A list of all balanced vectors $\mathbf{v} \in \Z^p_{+}$ whose sum is $n$, up to re-ordering of entries, for $1 \leq n \leq 9$. Compare with Table~\ref{tab:totally_balanced_smallp} and with the complete multipartite graphs listed in Table~\ref{tab:small_graphs}.}\label{tab:totally_balanced_smallsum}
\end{table}

\subsection{Graphs on nine or fewer vertices}\label{sec:small_graphs}

Here (in Table~\ref{tab:small_graphs}) we list all simple connected graphs on $1 \leq n \leq 9$ vertices that are $\{-1,0,1\}$- or $\{-1,1\}$-diagonalizable, and their corresponding bandwidths. These graphs and their diagonalizations were computed via brute-force computer search; details and code are provided at \cite{matlab}. In most cases when $n \in \{1,2\} \cup \{4k : k \in \mathbb{Z}_{+}\}$ the bandwidths are $1$, implying that the quasi-orthogonality condition of weakly Hadamard diagonalizable is typically only necessary for graphs of other sizes.

Most simple connected weakly Hadamard diagonalizable graphs on nine or fewer vertices were listed in \cite[Appendix~A]{adm2021weakly}. We use an asterisk ($*$) in the ``Name'' column of Table~\ref{tab:small_graphs} to identify the $\{-1,0,1\}$-diagonalizable graphs that were absent from the list appearing in \cite{adm2021weakly} (indicating that they have $\{-1,0,1\}$-bandwidth strictly greater than $2$, or they were missed in the computer search performed in \cite{adm2021weakly}). We note that \cite[Example~9]{MMP} showed that there are at least 23 nonisomorphic simple connected graphs on $8$ vertices that have $\{-1,0,1\}$-bandwidth equal to $1$, 6 of which are Hadamard diagonalizable (i.e., have $\{-1,1\}$-bandwidth equal to $1$), and the other 17 of which are not (since they are not regular). This agrees with the results displayed in Table~\ref{tab:small_graphs}, which also includes 3 additional graphs on eight vertices that have $\{-1,0,1\}$-bandwidth $2$.

We also note, however, that the small bandwidths displayed in Table~\ref{tab:small_graphs} could be an artifact of that fact that these graphs all have very few vertices. In particular, since eigenvectors coming from distinct eigenspaces of a graph's Laplacian matrix are necessarily orthogonal to each other, the largest finite bandwidth that a graph can have is the largest multiplicity of an eigenvalue of its Laplacian (which is typically a small integer when $n \leq 9$). A larger number of vertices would allow for larger eigenspaces, in which case there will likely be far more graphs with $\{-1,0,1\}$-bandwidth larger than $2$. 

\renewcommand\cellalign{cc}
\setlength{\LTcapwidth}{\textwidth}
\begin{longtable}{ c c c c c c c } 
    \toprule
    $n$ & Graph & Name &  $\{-1,0,1\}$-bandwidth & $\{-1,1\}$-bandwidth \\
    \midrule
    
    $1$ & \begin{tikzpicture}[scale=0.75,every node/.style={circle,fill=black,inner sep=2.5pt}]
        \node (1) at (0,0) {};
    \end{tikzpicture} & $K_1$  & $1$   & $1$ \\\midrule
    
    $2$ & \begin{tikzpicture}[scale=0.75,every node/.style={circle,fill=black,inner sep=2.5pt}]
        \node (1) at (0,0) {};
        \node (2) at (1,0) {};
        \draw (1) -- (2);
    \end{tikzpicture} & $K_2$ &  $1$ &   $1$ \\\midrule

    \makecell{$3$} & \makecell{\begin{tikzpicture}[scale=0.75,every node/.style={circle,fill=black,inner sep=2.5pt}]
        \node (1) at (0,0.866) {};
        \node (2) at (0.5,0) {};
        \node (3) at (-0.5,0) {};
        \draw (1) -- (2) -- (3) -- (1);
    \end{tikzpicture}} & \makecell{$K_3$}  & \makecell{$2$} &  \makecell{$\infty$} \\\midrule

    \makecell{$4$} & \makecell{\begin{tikzpicture}[scale=0.75,every node/.style={circle,fill=black,inner sep=2.5pt}]
        \node (1) at (-0.5,-0.5) {};
        \node (2) at (-0.5,0.5) {};
        \node (3) at (0.5,0.5) {};
        \node (4) at (0.5,-0.5) {};
        \draw (1) -- (3) -- (4) -- (2) -- (1);
    \end{tikzpicture}} & \makecell{$K_{2,2}$} &   \makecell{$1$} &   \makecell{$1$} \\
    
    & \makecell{\begin{tikzpicture}[scale=0.75,every node/.style={circle,fill=black,inner sep=2.5pt}]
        \node (1) at (-0.5,-0.5) {};
        \node (2) at (-0.5,0.5) {};
        \node (3) at (0.5,0.5) {};
        \node (4) at (0.5,-0.5) {};
        \draw (1) -- (3) -- (4) -- (2) -- (1);
        \draw (2) -- (3);
    \end{tikzpicture}} & \makecell{$K_{2,1,1}$}  & \makecell{$1$}   & \makecell{$\infty$} \\
    
    & \makecell{\begin{tikzpicture}[scale=0.75,every node/.style={circle,fill=black,inner sep=2.5pt}]
        \node (1) at (-0.5,-0.5) {};
        \node (2) at (-0.5,0.5) {};
        \node (3) at (0.5,0.5) {};
        \node (4) at (0.5,-0.5) {};
        \draw (1) -- (2) -- (3) -- (4) -- (1);
        \draw (1) -- (3);
        \draw (2) -- (4);
    \end{tikzpicture}} & \makecell{$K_4$}   & \makecell{$1$}  & \makecell{$1$} \\\midrule

    \makecell{$5$} & \makecell{\begin{tikzpicture}[scale=0.5,every node/.style={circle,fill=black,inner sep=2.5pt}]
        \node (1) at (0,1) {};
        \node (2) at (0.951,0.309) {};
        \node (3) at (-0.951,0.309) {};
        \node (4) at (0.588,-0.809) {};
        \node (5) at (-0.588,-0.809) {};
        \draw (1) -- (2) -- (3) -- (4);
        \draw (5) -- (1) -- (3) -- (5) -- (2) -- (4) -- (1);
    \end{tikzpicture}} & \makecell{$K_{2,1,1,1}$} & \makecell{$2$} &  \makecell{$\infty$} \\
    
    & \makecell{\begin{tikzpicture}[scale=0.5,every node/.style={circle,fill=black,inner sep=2.5pt}]
        \node (1) at (0,1) {};
        \node (2) at (0.951,0.309) {};
        \node (3) at (-0.951,0.309) {};
        \node (4) at (0.588,-0.809) {};
        \node (5) at (-0.588,-0.809) {};
        \draw (1) -- (2) -- (3) -- (4) -- (5) -- (1);
        \draw (1) -- (3) -- (5) -- (2) -- (4) -- (1);
    \end{tikzpicture}} & \makecell{$K_5$}   & \makecell{$2$} &  \makecell{$\infty$} \\\midrule

    \makecell{$6$} & \makecell{\begin{tikzpicture}[scale=0.5,every node/.style={circle,fill=black,inner sep=2.5pt}]
        \node (1) at (0.5,0.866) {};
        \node (2) at (1,0) {};
        \node (3) at (0.5,-0.866) {};
        \node (4) at (-0.5,-0.866) {};
        \node (5) at (-1,0) {};
        \node (6) at (-0.5,0.866) {};
        \draw (1) -- (2) -- (3) -- (4) -- (5) -- (6) -- (1);
    \end{tikzpicture}} & \makecell{$(K_3 \mathbin{\square} K_2)^{\textup{c}}$} &   \makecell{$2$} &  \makecell{$\infty$} \\
    
    & \makecell{\begin{tikzpicture}[scale=0.5,every node/.style={circle,fill=black,inner sep=2.5pt}]
        \node (1) at (0.5,0.866) {};
        \node (2) at (1,0) {};
        \node (3) at (0.5,-0.866) {};
        \node (4) at (-0.5,-0.866) {};
        \node (5) at (-1,0) {};
        \node (6) at (-0.5,0.866) {};
        \draw (1) -- (2) -- (3) -- (4) -- (5) -- (6) -- (1);
        \draw (1) -- (4);
        \draw (2) -- (5);
        \draw (3) -- (6);
    \end{tikzpicture}} & \makecell{$K_{3,3}$} &  \makecell{$2$}   & \makecell{$\infty$} \\
    
    & \makecell{\begin{tikzpicture}[scale=0.5,every node/.style={circle,fill=black,inner sep=2.5pt}]
        \node (1) at (0.5,0.866) {};
        \node (2) at (1,0) {};
        \node (3) at (0.5,-0.866) {};
        \node (4) at (-0.5,-0.866) {};
        \node (5) at (-1,0) {};
        \node (6) at (-0.5,0.866) {};
        \draw (1) -- (2) -- (3) -- (4) -- (5) -- (6) -- (1);
        \draw (1) -- (4);
        \draw (2) -- (5);
        \draw (3) -- (6);
        \draw (2) -- (4) -- (6) -- (2);
    \end{tikzpicture}} & \makecell{$K_{3,1,1,1}$}   & \makecell{$2$}  & \makecell{$\infty$} \\
    
    & \makecell{\begin{tikzpicture}[scale=0.5,every node/.style={circle,fill=black,inner sep=2.5pt}]
        \node (1) at (0.5,0.866) {};
        \node (2) at (1,0) {};
        \node (3) at (0.5,-0.866) {};
        \node (4) at (-0.5,-0.866) {};
        \node (5) at (-1,0) {};
        \node (6) at (-0.5,0.866) {};
        \draw (1) -- (4);
        \draw (2) -- (5);
        \draw (3) -- (6);
        \draw (2) -- (4) -- (6) -- (2);
        \draw (1) -- (3) -- (5) -- (1);
    \end{tikzpicture}} & \makecell{$K_3 \mathbin{\square} K_2$} &   \makecell{$2$} & \makecell{$\infty$} \\
    
    & \makecell{\begin{tikzpicture}[scale=0.5,every node/.style={circle,fill=black,inner sep=2.5pt}]
        \node (1) at (0.5,0.866) {};
        \node (2) at (1,0) {};
        \node (3) at (0.5,-0.866) {};
        \node (4) at (-0.5,-0.866) {};
        \node (5) at (-1,0) {};
        \node (6) at (-0.5,0.866) {};
        \draw (1) -- (2) -- (3);
        \draw (4) -- (5) -- (6);
        \draw (1) -- (4);
        \draw (3) -- (6);
        \draw (2) -- (4) -- (6) -- (2);
        \draw (1) -- (3) -- (5) -- (1);
    \end{tikzpicture}} & \makecell{$K_{2,2,2}$} &   \makecell{$2$} &   \makecell{$\infty$} \\
    
    & \makecell{\begin{tikzpicture}[scale=0.5,every node/.style={circle,fill=black,inner sep=2.5pt}]
        \node (1) at (0.5,0.866) {};
        \node (2) at (1,0) {};
        \node (3) at (0.5,-0.866) {};
        \node (4) at (-0.5,-0.866) {};
        \node (5) at (-1,0) {};
        \node (6) at (-0.5,0.866) {};
        \draw (1) -- (2) -- (3);
        \draw (4) -- (5) -- (6);
        \draw (1) -- (4);
        \draw (2) -- (5);
        \draw (3) -- (6);
        \draw (2) -- (4) -- (6) -- (2);
        \draw (1) -- (3) -- (5) -- (1);
    \end{tikzpicture}} & \makecell{$K_{2,2,1,1}$} & \makecell{$2$}  & \makecell{$\infty$} \\
    
    & \makecell{\begin{tikzpicture}[scale=0.5,every node/.style={circle,fill=black,inner sep=2.5pt}]
        \node (1) at (0.5,0.866) {};
        \node (2) at (1,0) {};
        \node (3) at (0.5,-0.866) {};
        \node (4) at (-0.5,-0.866) {};
        \node (5) at (-1,0) {};
        \node (6) at (-0.5,0.866) {};
        \draw (1) -- (2) -- (3);
        \draw (4) -- (5) -- (6) -- (1);
        \draw (1) -- (4);
        \draw (2) -- (5);
        \draw (3) -- (6);
        \draw (2) -- (4) -- (6) -- (2);
        \draw (1) -- (3) -- (5) -- (1);
    \end{tikzpicture}} & \makecell{$K_{2,1,1,1,1}$} &   \makecell{$2$} &  \makecell{$\infty$} \\
    
    & \makecell{\begin{tikzpicture}[scale=0.5,every node/.style={circle,fill=black,inner sep=2.5pt}]
        \node (1) at (0.5,0.866) {};
        \node (2) at (1,0) {};
        \node (3) at (0.5,-0.866) {};
        \node (4) at (-0.5,-0.866) {};
        \node (5) at (-1,0) {};
        \node (6) at (-0.5,0.866) {};
        \draw (1) -- (2) -- (3) -- (4) -- (5) -- (6) -- (1);
        \draw (1) -- (4);
        \draw (2) -- (5);
        \draw (3) -- (6);
        \draw (2) -- (4) -- (6) -- (2);
        \draw (1) -- (3) -- (5) -- (1);
    \end{tikzpicture}} & \makecell{$K_6$}  & \makecell{$2$} & \makecell{$5$} \\\midrule

    \makecell{$7$} & \makecell{\begin{tikzpicture}[scale=0.5,every node/.style={circle,fill=black,inner sep=2.2pt}]
        \node (1) at (0,1) {};
        \node (2) at (0.782,0.623) {};
        \node (3) at (0.975,-0.223) {};
        \node (4) at (0.434,-0.901) {};
        \node (5) at (-0.434,-0.901) {};
        \node (6) at (-0.975,-0.223) {};
        \node (7) at (-0.782,0.623) {};
        \draw (2) -- (3) -- (4);
        \draw (5) -- (6) -- (7);
        \draw (2) -- (4) -- (6) -- (1) -- (3) -- (5) -- (7) -- (3) -- (6) -- (2) -- (5) -- (1) -- (4) -- (7);
    \end{tikzpicture}} & \makecell{$K_{3,2,1,1}$}  & \makecell{$2$} & \makecell{$\infty$} \\
    
    & \makecell{\begin{tikzpicture}[scale=0.5,every node/.style={circle,fill=black,inner sep=2.2pt}]
        \node (1) at (0,1) {};
        \node (2) at (0.782,0.623) {};
        \node (3) at (0.975,-0.223) {};
        \node (4) at (0.434,-0.901) {};
        \node (5) at (-0.434,-0.901) {};
        \node (6) at (-0.975,-0.223) {};
        \node (7) at (-0.782,0.623) {};
        \draw (2) -- (3) -- (4) -- (5) -- (6) -- (7);
        \draw (2) -- (4) -- (6) -- (1) -- (3) -- (5) -- (7) -- (3) -- (6) -- (2) -- (5) -- (1) -- (4) -- (7);
    \end{tikzpicture}} & \makecell{$K_{3,1,1,1,1}$} &  \makecell{$2$} &  \makecell{$\infty$} \\
    
    & \makecell{\begin{tikzpicture}[scale=0.5,every node/.style={circle,fill=black,inner sep=2.2pt}]
        \node (1) at (0,1) {};
        \node (2) at (0.782,0.623) {};
        \node (3) at (0.975,-0.223) {};
        \node (4) at (0.434,-0.901) {};
        \node (5) at (-0.434,-0.901) {};
        \node (6) at (-0.975,-0.223) {};
        \node (7) at (-0.782,0.623) {};
        \draw (1) -- (2);
        \draw (3) -- (4) -- (5) -- (6);
        \draw (7) -- (1) -- (3) -- (5) -- (7) -- (2) -- (4) -- (6) -- (1) -- (4) -- (7) -- (3) -- (6) -- (2) -- (5) -- (1);
    \end{tikzpicture}} & \makecell{$K_{2,2,1,1,1}$}   & \makecell{$2$} &  \makecell{$\infty$} \\
    
    & \makecell{\begin{tikzpicture}[scale=0.5,every node/.style={circle,fill=black,inner sep=2.2pt}]
        \node (1) at (0,1) {};
        \node (2) at (0.782,0.623) {};
        \node (3) at (0.975,-0.223) {};
        \node (4) at (0.434,-0.901) {};
        \node (5) at (-0.434,-0.901) {};
        \node (6) at (-0.975,-0.223) {};
        \node (7) at (-0.782,0.623) {};
        \draw (1) -- (2) -- (3) -- (4);
        \draw (5) -- (6) -- (7) -- (1) -- (3) -- (5) -- (7) -- (2) -- (4) -- (6) -- (1) -- (4) -- (7) -- (3) -- (6) -- (2) -- (5) -- (1);
    \end{tikzpicture}} & \makecell{$K_{2,1,1,1,1,1}$} &  \makecell{$2$} &  \makecell{$\infty$} \\
    
    & \makecell{\begin{tikzpicture}[scale=0.5,every node/.style={circle,fill=black,inner sep=2.2pt}]
        \node (1) at (0,1) {};
        \node (2) at (0.782,0.623) {};
        \node (3) at (0.975,-0.223) {};
        \node (4) at (0.434,-0.901) {};
        \node (5) at (-0.434,-0.901) {};
        \node (6) at (-0.975,-0.223) {};
        \node (7) at (-0.782,0.623) {};
        \draw (1) -- (2) -- (3) -- (4) -- (5) -- (6) -- (7) -- (1) -- (3) -- (5) -- (7) -- (2) -- (4) -- (6) -- (1) -- (4) -- (7) -- (3) -- (6) -- (2) -- (5) -- (1);
    \end{tikzpicture}} & \makecell{$K_7$}  & \makecell{$2$}  & \makecell{$\infty$} \\\midrule

    \makecell{$8$} & \makecell{\begin{tikzpicture}[scale=0.55,every node/.style={circle,fill=black,inner sep=2pt}]
        \node (1) at (0.383,0.924) {};
        \node (2) at (0.924,0.383) {};
        \node (3) at (0.924,-0.383) {};
        \node (4) at (0.383,-0.924) {};
        \node (5) at (-0.383,-0.924) {};
        \node (6) at (-0.924,-0.383) {};
        \node (7) at (-0.924,0.383) {};
        \node (8) at (-0.383,0.924) {};
        \draw (1) -- (5);\draw (1) -- (6);\draw (1) -- (7);
        \draw (2) -- (5);\draw (2) -- (6);\draw (2) -- (8);
        \draw (3) -- (5);\draw (3) -- (7);\draw (3) -- (8);
        \draw (4) -- (6);\draw (4) -- (7);\draw (4) -- (8);
    \end{tikzpicture}} & \makecell{$(K_4 \mathbin{\square} K_2)^{\textup{c}}$} &   \makecell{$1$}  & \makecell{$1$} \\
    
    & \makecell{\begin{tikzpicture}[scale=0.55,every node/.style={circle,fill=black,inner sep=2pt}]
        \node (1) at (0.383,0.924) {};
        \node (2) at (0.924,0.383) {};
        \node (3) at (0.924,-0.383) {};
        \node (4) at (0.383,-0.924) {};
        \node (5) at (-0.383,-0.924) {};
        \node (6) at (-0.924,-0.383) {};
        \node (7) at (-0.924,0.383) {};
        \node (8) at (-0.383,0.924) {};
        \draw (1) -- (5) -- (2) -- (6) -- (1) -- (7) -- (2) -- (8) -- (1);
        \draw (3) -- (5) -- (4) -- (6) -- (3) -- (7) -- (4) -- (8) -- (3);
    \end{tikzpicture}} & \makecell{$K_{4,4}$}   & \makecell{$1$}  & \makecell{$1$} \\
    
    & \makecell{\begin{tikzpicture}[scale=0.55,every node/.style={circle,fill=black,inner sep=2pt}]
        \node (1) at (0.383,0.924) {};
        \node (2) at (0.924,0.383) {};
        \node (3) at (0.924,-0.383) {};
        \node (4) at (0.383,-0.924) {};
        \node (5) at (-0.383,-0.924) {};
        \node (6) at (-0.924,-0.383) {};
        \node (7) at (-0.924,0.383) {};
        \node (8) at (-0.383,0.924) {};
        \draw (1) -- (5) -- (2) -- (6) -- (1) -- (7) -- (2) -- (8) -- (1);
        \draw (3) -- (5) -- (4) -- (6) -- (3) -- (7) -- (4) -- (8) -- (3);
        \draw (6) -- (7);
    \end{tikzpicture}} & \makecell{$(K_4 \sqcup K_{2,1,1})^{\textup{c}}$}  & \makecell{$1$}  & \makecell{$\infty$} \\
    
    & \makecell{\begin{tikzpicture}[scale=0.55,every node/.style={circle,fill=black,inner sep=2pt}]
        \node (1) at (0.383,0.924) {};
        \node (2) at (0.924,0.383) {};
        \node (3) at (0.924,-0.383) {};
        \node (4) at (0.383,-0.924) {};
        \node (5) at (-0.383,-0.924) {};
        \node (6) at (-0.924,-0.383) {};
        \node (7) at (-0.924,0.383) {};
        \node (8) at (-0.383,0.924) {};
        \draw (1) -- (5) -- (2) -- (6) -- (1) -- (7) -- (2) -- (8) -- (1);
        \draw (3) -- (5) -- (4) -- (6) -- (3) -- (7) -- (4) -- (8) -- (3);
        \draw (5) -- (6);\draw (7) -- (8);
    \end{tikzpicture}} & \makecell{$(K_4 \sqcup K_{2,2})^{\textup{c}}$} & \makecell{$1$}  & \makecell{$\infty$} \\
    
    & \makecell{\begin{tikzpicture}[scale=0.55,every node/.style={circle,fill=black,inner sep=2pt}]
        \node (1) at (0.383,0.924) {};
        \node (2) at (0.924,0.383) {};
        \node (3) at (0.924,-0.383) {};
        \node (4) at (0.383,-0.924) {};
        \node (5) at (-0.383,-0.924) {};
        \node (6) at (-0.924,-0.383) {};
        \node (7) at (-0.924,0.383) {};
        \node (8) at (-0.383,0.924) {};
        \draw (1) -- (5) -- (2) -- (6) -- (1) -- (7) -- (2) -- (8) -- (1);
        \draw (3) -- (5) -- (4) -- (6) -- (3) -- (7) -- (4) -- (8) -- (3);
        \draw (5) -- (7) -- (6) -- (8) -- (5);
    \end{tikzpicture}} & \makecell{$K_{4,2,2}$} &  \makecell{$1$} &   \makecell{$\infty$} \\
    
    & \makecell{\begin{tikzpicture}[scale=0.55,every node/.style={circle,fill=black,inner sep=2pt}]
        \node (1) at (0.383,0.924) {};
        \node (2) at (0.924,0.383) {};
        \node (3) at (0.924,-0.383) {};
        \node (4) at (0.383,-0.924) {};
        \node (5) at (-0.383,-0.924) {};
        \node (6) at (-0.924,-0.383) {};
        \node (7) at (-0.924,0.383) {};
        \node (8) at (-0.383,0.924) {};
        \draw (1) -- (5) -- (2) -- (6) -- (1) -- (7) -- (2) -- (8) -- (1);
        \draw (3) -- (5) -- (4) -- (6) -- (3) -- (7) -- (4) -- (8) -- (3);
        \draw (5) -- (7) -- (6) -- (8) -- (5);\draw (7) -- (8);
    \end{tikzpicture}} & \makecell{$K_{4,2,1,1}$} &   \makecell{$1$} &  \makecell{$\infty$} \\
    
    & \makecell{\begin{tikzpicture}[scale=0.55,every node/.style={circle,fill=black,inner sep=2pt}]
        \node (1) at (0.383,0.924) {};
        \node (2) at (0.924,0.383) {};
        \node (3) at (0.924,-0.383) {};
        \node (4) at (0.383,-0.924) {};
        \node (5) at (-0.383,-0.924) {};
        \node (6) at (-0.924,-0.383) {};
        \node (7) at (-0.924,0.383) {};
        \node (8) at (-0.383,0.924) {};
        \draw (1) -- (5) -- (2) -- (6) -- (1) -- (7) -- (2) -- (8) -- (1);
        \draw (3) -- (5) -- (4) -- (6) -- (3) -- (7) -- (4) -- (8) -- (3);
        \draw (5) -- (7) -- (6) -- (8) -- (5);\draw (7) -- (8);\draw (5) -- (6);
    \end{tikzpicture}} & \makecell{$K_{4,1,1,1,1}$} &   \makecell{$1$} &  \makecell{$\infty$} \\
    
    & \makecell{\begin{tikzpicture}[scale=0.55,every node/.style={circle,fill=black,inner sep=2pt}]
        \node (1) at (0.383,0.924) {};
        \node (2) at (0.924,0.383) {};
        \node (3) at (0.924,-0.383) {};
        \node (4) at (0.383,-0.924) {};
        \node (5) at (-0.383,-0.924) {};
        \node (6) at (-0.924,-0.383) {};
        \node (7) at (-0.924,0.383) {};
        \node (8) at (-0.383,0.924) {};
        \draw (1) -- (5) -- (2) -- (6) -- (1) -- (7) -- (2) -- (8) -- (1);
        \draw (3) -- (5) -- (4) -- (6) -- (3) -- (7) -- (4) -- (8) -- (3);
        \draw (6) -- (7);\draw (2) -- (3);
    \end{tikzpicture}} & \makecell{$(K_{2,1,1} \sqcup K_{2,1,1})^{\textup{c}}$} & \makecell{$1$}& \makecell{$\infty$} \\
    
    & \makecell{\begin{tikzpicture}[scale=0.55,every node/.style={circle,fill=black,inner sep=2pt}]
        \node (1) at (0.383,0.924) {};
        \node (2) at (0.924,0.383) {};
        \node (3) at (0.924,-0.383) {};
        \node (4) at (0.383,-0.924) {};
        \node (5) at (-0.383,-0.924) {};
        \node (6) at (-0.924,-0.383) {};
        \node (7) at (-0.924,0.383) {};
        \node (8) at (-0.383,0.924) {};
        \draw (1) -- (5) -- (2) -- (6) -- (1) -- (7) -- (2) -- (8) -- (1);
        \draw (3) -- (5) -- (4) -- (6) -- (3) -- (7) -- (4) -- (8) -- (3);
        \draw (5) -- (7) -- (6) -- (8) -- (5);\draw (2) -- (3);
    \end{tikzpicture}} & \makecell{$(K_{2,1,1} \sqcup K_{2} \sqcup K_{2})^{\textup{c}}$} & \makecell{$1$} & \makecell{$\infty$} \\
    
    & \makecell{\begin{tikzpicture}[scale=0.55,every node/.style={circle,fill=black,inner sep=2pt}]
        \node (1) at (0.383,0.924) {};
        \node (2) at (0.924,0.383) {};
        \node (3) at (0.924,-0.383) {};
        \node (4) at (0.383,-0.924) {};
        \node (5) at (-0.383,-0.924) {};
        \node (6) at (-0.924,-0.383) {};
        \node (7) at (-0.924,0.383) {};
        \node (8) at (-0.383,0.924) {};
        \draw (1) -- (5) -- (2) -- (6) -- (1) -- (7) -- (2) -- (8) -- (1);
        \draw (3) -- (5) -- (4) -- (6) -- (3) -- (7) -- (4) -- (8) -- (3);
        \draw (5) -- (7) -- (6) -- (8) -- (5);\draw (7) -- (8);\draw (2) -- (3);
    \end{tikzpicture}} & \makecell{$(K_{2,1,1} \sqcup K_{2} \sqcup K_1 \sqcup K_1)^{\textup{c}}$,*} &   \makecell{$1$} &   \makecell{$\infty$} \\
    
    & \makecell{\begin{tikzpicture}[scale=0.55,every node/.style={circle,fill=black,inner sep=2pt}]
        \node (1) at (0.383,0.924) {};
        \node (2) at (0.924,0.383) {};
        \node (3) at (0.924,-0.383) {};
        \node (4) at (0.383,-0.924) {};
        \node (5) at (-0.383,-0.924) {};
        \node (6) at (-0.924,-0.383) {};
        \node (7) at (-0.924,0.383) {};
        \node (8) at (-0.383,0.924) {};
        \draw (1) -- (5) -- (2) -- (6) -- (1) -- (7) -- (2) -- (8) -- (1);
        \draw (3) -- (5) -- (4) -- (6) -- (3) -- (7) -- (4) -- (8) -- (3);
        \draw (5) -- (7) -- (6) -- (8) -- (5);\draw (7) -- (8);\draw (5) -- (6);\draw (2) -- (3);
    \end{tikzpicture}} & \makecell{$(K_{2,1,1} \sqcup K_1 \sqcup K_1 \sqcup K_1 \sqcup K_1)^{\textup{c}}$} & \makecell{$1$}   & \makecell{$\infty$} \\
    
    & \makecell{\begin{tikzpicture}[scale=0.55,every node/.style={circle,fill=black,inner sep=2pt}]
        \node (1) at (0.383,0.924) {};
        \node (2) at (0.924,0.383) {};
        \node (3) at (0.924,-0.383) {};
        \node (4) at (0.383,-0.924) {};
        \node (5) at (-0.383,-0.924) {};
        \node (6) at (-0.924,-0.383) {};
        \node (7) at (-0.924,0.383) {};
        \node (8) at (-0.383,0.924) {};
        \draw (1) -- (4) -- (7) -- (1);
        \draw (1) -- (5) -- (3) -- (8) -- (5);
        \draw (7) -- (2) -- (4) -- (8) -- (6);
        \draw (2) -- (6) -- (3) -- (7);
    \end{tikzpicture}} & \makecell{$K_{2,1,1} \mathbin{\square} K_2$,*}  & \makecell{$1$} & \makecell{$\infty$} \\
    
    & \makecell{\begin{tikzpicture}[scale=0.55,every node/.style={circle,fill=black,inner sep=2pt}]
        \node (1) at (0.383,0.924) {};
        \node (2) at (0.924,0.383) {};
        \node (3) at (0.924,-0.383) {};
        \node (4) at (0.383,-0.924) {};
        \node (5) at (-0.383,-0.924) {};
        \node (6) at (-0.924,-0.383) {};
        \node (7) at (-0.924,0.383) {};
        \node (8) at (-0.383,0.924) {};
        \draw (2) -- (5) -- (1) -- (4) -- (2);
        \draw (2) -- (6) -- (1) -- (7) -- (2);
        \draw (1) -- (8) -- (4);
        \draw (5) -- (3) -- (4);
        \draw (4) -- (6) -- (3) -- (7) -- (5);
        \draw (5) -- (8) -- (6);
        \draw (7) -- (8);
    \end{tikzpicture}} & \makecell{$(K_{2,2} \sqcup K_{2,1,1})^{\textup{c}}$} & \makecell{$1$}  & \makecell{$\infty$} \\
    
    & \makecell{\begin{tikzpicture}[scale=0.55,every node/.style={circle,fill=black,inner sep=2pt}]
        \node (1) at (0.383,0.924) {};
        \node (2) at (0.924,0.383) {};
        \node (3) at (0.924,-0.383) {};
        \node (4) at (0.383,-0.924) {};
        \node (5) at (-0.383,-0.924) {};
        \node (6) at (-0.924,-0.383) {};
        \node (7) at (-0.924,0.383) {};
        \node (8) at (-0.383,0.924) {};
        \draw (3) -- (4);
        \draw (5) -- (6);
        \draw (7) -- (8) -- (1);
        \draw (1) -- (4) -- (7) -- (2) -- (5) -- (8) -- (3) -- (6) -- (1);
        \draw (3) -- (5) -- (7) -- (1) -- (5);
        \draw (2) -- (4) -- (6) -- (8) -- (2) -- (6);
        \draw (3) -- (7);
        \draw (4) -- (8);
    \end{tikzpicture}} & \makecell{$K_{3,2,2,1}$} & \makecell{$2$} &  \makecell{$\infty$} \\
    
    & \makecell{\begin{tikzpicture}[scale=0.55,every node/.style={circle,fill=black,inner sep=2pt}]
        \node (1) at (0.383,0.924) {};
        \node (2) at (0.924,0.383) {};
        \node (3) at (0.924,-0.383) {};
        \node (4) at (0.383,-0.924) {};
        \node (5) at (-0.383,-0.924) {};
        \node (6) at (-0.924,-0.383) {};
        \node (7) at (-0.924,0.383) {};
        \node (8) at (-0.383,0.924) {};
        \draw (3) -- (4);
        \draw (5) -- (6);
        \draw (7) -- (8) -- (1);
        \draw (1) -- (4) -- (7) -- (2) -- (5) -- (8) -- (3) -- (6) -- (1);
        \draw (3) -- (5) -- (7) -- (1) -- (5);
        \draw (2) -- (4) -- (6) -- (8) -- (2) -- (6);
        \draw (3) -- (7) -- (6);
        \draw (4) -- (8);
    \end{tikzpicture}} & \makecell{$K_{3,2,1,1,1}$} &  \makecell{$2$} &   \makecell{$\infty$} \\
    
    & \makecell{\begin{tikzpicture}[scale=0.55,every node/.style={circle,fill=black,inner sep=2pt}]
        \node (1) at (0.383,0.924) {};
        \node (2) at (0.924,0.383) {};
        \node (3) at (0.924,-0.383) {};
        \node (4) at (0.383,-0.924) {};
        \node (5) at (-0.383,-0.924) {};
        \node (6) at (-0.924,-0.383) {};
        \node (7) at (-0.924,0.383) {};
        \node (8) at (-0.383,0.924) {};
        \draw (3) -- (4) -- (5) -- (6);
        \draw (7) -- (8) -- (1);
        \draw (1) -- (4) -- (7) -- (2) -- (5) -- (8) -- (3) -- (6) -- (1);
        \draw (3) -- (5) -- (7) -- (1) -- (5);
        \draw (2) -- (4) -- (6) -- (8) -- (2) -- (6);
        \draw (3) -- (7) -- (6);
        \draw (4) -- (8);
    \end{tikzpicture}} & \makecell{$K_{3,1,1,1,1,1}$}  & \makecell{$2$}   & \makecell{$\infty$} \\
    
    & \makecell{\begin{tikzpicture}[scale=0.55,every node/.style={circle,fill=black,inner sep=2pt}]
        \node (1) at (0.383,0.924) {};
        \node (2) at (0.924,0.383) {};
        \node (3) at (0.924,-0.383) {};
        \node (4) at (0.383,-0.924) {};
        \node (5) at (-0.383,-0.924) {};
        \node (6) at (-0.924,-0.383) {};
        \node (7) at (-0.924,0.383) {};
        \node (8) at (-0.383,0.924) {};
        \draw (3) -- (8) -- (4) -- (7) -- (3);
        \draw (6) -- (1) -- (3) -- (5) -- (7) -- (1) -- (5) -- (2) -- (4) -- (6) -- (8) -- (2) -- (6);
    \end{tikzpicture}} & \makecell{$K_4 \mathbin{\square} K_2$} &  \makecell{$1$}  & \makecell{$1$} \\
    
    & \makecell{\begin{tikzpicture}[scale=0.55,every node/.style={circle,fill=black,inner sep=2pt}]
        \node (1) at (0.383,0.924) {};
        \node (2) at (0.924,0.383) {};
        \node (3) at (0.924,-0.383) {};
        \node (4) at (0.383,-0.924) {};
        \node (5) at (-0.383,-0.924) {};
        \node (6) at (-0.924,-0.383) {};
        \node (7) at (-0.924,0.383) {};
        \node (8) at (-0.383,0.924) {};
        \draw (4) -- (7) -- (2) -- (5) -- (4) -- (8) -- (1) -- (3) -- (5) -- (7) -- (1) -- (5);
        \draw (8) -- (3) -- (6) -- (1);
        \draw (2) -- (4) -- (6) -- (8) -- (2) -- (6);
        \draw (3) -- (7);
    \end{tikzpicture}} & \makecell{$(K_{2,2} \sqcup K_{2,2})^{\textup{c}}$} & \makecell{$1$} & \makecell{$1$} \\
    
    & \makecell{\begin{tikzpicture}[scale=0.55,every node/.style={circle,fill=black,inner sep=2pt}]
        \node (1) at (0.383,0.924) {};
        \node (2) at (0.924,0.383) {};
        \node (3) at (0.924,-0.383) {};
        \node (4) at (0.383,-0.924) {};
        \node (5) at (-0.383,-0.924) {};
        \node (6) at (-0.924,-0.383) {};
        \node (7) at (-0.924,0.383) {};
        \node (8) at (-0.383,0.924) {};
        \draw (4) -- (7) -- (2) -- (5) -- (4) -- (8) -- (1) -- (3) -- (5) -- (7) -- (1) -- (5) -- (8) -- (3) -- (6) -- (1);
        \draw (2) -- (4) -- (6) -- (8) -- (2) -- (6) -- (7) -- (3);
    \end{tikzpicture}} & \makecell{$(K_{2,2} \sqcup K_{2} \sqcup K_2)^{\textup{c}}$} & \makecell{$1$}   & \makecell{$\infty$} \\
    
    & \makecell{\begin{tikzpicture}[scale=0.55,every node/.style={circle,fill=black,inner sep=2pt}]
        \node (1) at (0.383,0.924) {};
        \node (2) at (0.924,0.383) {};
        \node (3) at (0.924,-0.383) {};
        \node (4) at (0.383,-0.924) {};
        \node (5) at (-0.383,-0.924) {};
        \node (6) at (-0.924,-0.383) {};
        \node (7) at (-0.924,0.383) {};
        \node (8) at (-0.383,0.924) {};
        \draw (4) -- (7) -- (2) -- (5) -- (4) -- (8) -- (1) -- (3) -- (5) -- (7) -- (1) -- (5) -- (8) -- (3) -- (6) -- (1);
        \draw (2) -- (4) -- (6) -- (8) -- (2) -- (6) -- (7) -- (3);
        \draw (7) -- (8);
    \end{tikzpicture}} & \makecell{$(K_{2,2} \sqcup K_{2} \sqcup K_1 \sqcup K_1)^{\textup{c}}$} & \makecell{$1$}  & \makecell{$\infty$} \\
    
    & \makecell{\begin{tikzpicture}[scale=0.55,every node/.style={circle,fill=black,inner sep=2pt}]
        \node (1) at (0.383,0.924) {};
        \node (2) at (0.924,0.383) {};
        \node (3) at (0.924,-0.383) {};
        \node (4) at (0.383,-0.924) {};
        \node (5) at (-0.383,-0.924) {};
        \node (6) at (-0.924,-0.383) {};
        \node (7) at (-0.924,0.383) {};
        \node (8) at (-0.383,0.924) {};
        \draw (4) -- (7) -- (2) -- (5) -- (4) -- (8) -- (1) -- (3) -- (5) -- (7) -- (1) -- (5) -- (8) -- (3) -- (6) -- (1);
        \draw (2) -- (4) -- (6) -- (8) -- (2) -- (6) -- (7) -- (3);
        \draw (7) -- (8);
        \draw (5) -- (6);
    \end{tikzpicture}} & \makecell{$(K_{2,2} \sqcup K_1 \sqcup K_1 \sqcup K_1 \sqcup K_1)^{\textup{c}}$} & \makecell{$1$}  & \makecell{$\infty$} \\
    
    & \makecell{\begin{tikzpicture}[scale=0.55,every node/.style={circle,fill=black,inner sep=2pt}]
        \node (1) at (0.383,0.924) {};
        \node (2) at (0.924,0.383) {};
        \node (3) at (0.924,-0.383) {};
        \node (4) at (0.383,-0.924) {};
        \node (5) at (-0.383,-0.924) {};
        \node (6) at (-0.924,-0.383) {};
        \node (7) at (-0.924,0.383) {};
        \node (8) at (-0.383,0.924) {};
        \draw (7) -- (3) -- (4);
        \draw (2) -- (1) -- (4) -- (7) -- (2) -- (5) -- (8) -- (3) -- (6) -- (1) -- (3) -- (5) -- (7) -- (1) -- (5);
        \draw (4) -- (8) -- (7);
        \draw (2) -- (4) -- (6) -- (8) -- (2) -- (6) -- (5);
    \end{tikzpicture}} & \makecell{$K_{2,2,2,2}$} &  \makecell{$1$} & \makecell{$1$} \\
    
    & \makecell{\begin{tikzpicture}[scale=0.55,every node/.style={circle,fill=black,inner sep=2pt}]
        \node (1) at (0.383,0.924) {};
        \node (2) at (0.924,0.383) {};
        \node (3) at (0.924,-0.383) {};
        \node (4) at (0.383,-0.924) {};
        \node (5) at (-0.383,-0.924) {};
        \node (6) at (-0.924,-0.383) {};
        \node (7) at (-0.924,0.383) {};
        \node (8) at (-0.383,0.924) {};
        \draw (2) -- (3) -- (4);
        \draw (5) -- (6) -- (7);
        \draw (8) -- (1) -- (4) -- (7) -- (2) -- (5) -- (8) -- (3) -- (6) -- (1);
        \draw (1) -- (3) -- (5) -- (7) -- (1) -- (5);
        \draw (2) -- (4) -- (6) -- (8) -- (2) -- (6);
        \draw (3) -- (7);
        \draw (4) -- (8);
    \end{tikzpicture}} & \makecell{$K_{2,2,2,1,1}$}   & \makecell{$1$}  & \makecell{$\infty$} \\
    
    & \makecell{\begin{tikzpicture}[scale=0.55,every node/.style={circle,fill=black,inner sep=2pt}]
        \node (1) at (0.383,0.924) {};
        \node (2) at (0.924,0.383) {};
        \node (3) at (0.924,-0.383) {};
        \node (4) at (0.383,-0.924) {};
        \node (5) at (-0.383,-0.924) {};
        \node (6) at (-0.924,-0.383) {};
        \node (7) at (-0.924,0.383) {};
        \node (8) at (-0.383,0.924) {};
        \draw (3) -- (4);
        \draw (1) -- (2) -- (3) -- (7) -- (6);
        \draw (1) -- (4) -- (7) -- (2) -- (5) -- (8) -- (3) -- (6) -- (1) -- (3) -- (5) -- (7) -- (1) -- (5);
        \draw (4) -- (8) -- (7);
        \draw (2) -- (4) -- (6) -- (8) -- (2) -- (6) -- (5);
    \end{tikzpicture}} & \makecell{$K_{2,2,1,1,1,1}$}   & \makecell{$1$}   & \makecell{$\infty$} \\
    
    & \makecell{\begin{tikzpicture}[scale=0.55,every node/.style={circle,fill=black,inner sep=2pt}]
        \node (1) at (0.383,0.924) {};
        \node (2) at (0.924,0.383) {};
        \node (3) at (0.924,-0.383) {};
        \node (4) at (0.383,-0.924) {};
        \node (5) at (-0.383,-0.924) {};
        \node (6) at (-0.924,-0.383) {};
        \node (7) at (-0.924,0.383) {};
        \node (8) at (-0.383,0.924) {};
        \draw (3) -- (4);
        \draw (1) -- (2) -- (3) -- (7) -- (6);
        \draw (8) -- (1) -- (4) -- (7) -- (2) -- (5) -- (8) -- (3) -- (6) -- (1) -- (3) -- (5) -- (7) -- (1) -- (5);
        \draw (4) -- (8) -- (7);
        \draw (2) -- (4) -- (6) -- (8) -- (2) -- (6) -- (5);
    \end{tikzpicture}} & \makecell{$K_{2,1,1,1,1,1,1}$}   & \makecell{$1$} &  \makecell{$\infty$} \\
    
    & \makecell{\begin{tikzpicture}[scale=0.55,every node/.style={circle,fill=black,inner sep=2pt}]
        \node (1) at (0.383,0.924) {};
        \node (2) at (0.924,0.383) {};
        \node (3) at (0.924,-0.383) {};
        \node (4) at (0.383,-0.924) {};
        \node (5) at (-0.383,-0.924) {};
        \node (6) at (-0.924,-0.383) {};
        \node (7) at (-0.924,0.383) {};
        \node (8) at (-0.383,0.924) {};
        \draw (1) -- (2) -- (3) -- (4) -- (5) -- (6) -- (7) -- (8) -- (1);
        \draw (1) -- (4) -- (7) -- (2) -- (5) -- (8) -- (3) -- (6) -- (1);
        \draw (1) -- (3) -- (5) -- (7) -- (1) -- (5);
        \draw (2) -- (4) -- (6) -- (8) -- (2) -- (6);
        \draw (3) -- (7);
        \draw (4) -- (8);
    \end{tikzpicture}} & \makecell{$K_8$}   & \makecell{$1$}   & \makecell{$1$} \\\midrule

    \makecell{$9$} & \makecell{\begin{tikzpicture}[scale=0.55,every node/.style={circle,fill=black,inner sep=2pt}]
        \node (1) at (0,1) {};
        \node (2) at (0.642,0.766) {};
        \node (3) at (0.985,0.174) {};
        \node (4) at (0.866,-0.5) {};
        \node (5) at (0.342,-0.940) {};
        \node (6) at (-0.342,-0.940) {};
        \node (7) at (-0.866,-0.5) {};
        \node (8) at (-0.985,0.174) {};
        \node (9) at (-0.642,0.766) {};
        \draw (6) -- (7) -- (8) -- (9) -- (1);
        \draw (5) -- (4) -- (6) -- (8) -- (1);
        \draw (4) -- (7) -- (1) -- (5) -- (9) -- (4) -- (8) -- (3) -- (7) -- (2) -- (6) -- (1);
        \draw (3) -- (6) -- (9) -- (3) -- (5) -- (7) -- (9) -- (2) -- (5) -- (8) -- (2);
    \end{tikzpicture}} & \makecell{$K_{4,2,1,1,1}$, *} & \makecell{$3$} &   \makecell{$\infty$} \\

    & \makecell{\begin{tikzpicture}[scale=0.55,every node/.style={circle,fill=black,inner sep=2pt}]
        \node (1) at (0,1) {};
        \node (2) at (0.642,0.766) {};
        \node (3) at (0.985,0.174) {};
        \node (4) at (0.866,-0.5) {};
        \node (5) at (0.342,-0.940) {};
        \node (6) at (-0.342,-0.940) {};
        \node (7) at (-0.866,-0.5) {};
        \node (8) at (-0.985,0.174) {};
        \node (9) at (-0.642,0.766) {};
        \draw (6) -- (7) -- (8) -- (9) -- (1);
        \draw (6) -- (5) -- (4) -- (6) -- (8) -- (1);
        \draw (4) -- (7) -- (1) -- (5) -- (9) -- (4) -- (8) -- (3) -- (7) -- (2) -- (6) -- (1);
        \draw (3) -- (6) -- (9) -- (3) -- (5) -- (7) -- (9) -- (2) -- (5) -- (8) -- (2);
    \end{tikzpicture}} & \makecell{$K_{4,1,1,1,1,1}$} & \makecell{$2$} &  \makecell{$\infty$} \\

    & \makecell{\begin{tikzpicture}[scale=0.55,every node/.style={circle,fill=black,inner sep=2pt}]
        \node (1) at (0,1) {};
        \node (2) at (0.642,0.766) {};
        \node (3) at (0.985,0.174) {};
        \node (4) at (0.866,-0.5) {};
        \node (5) at (0.342,-0.940) {};
        \node (6) at (-0.342,-0.940) {};
        \node (7) at (-0.866,-0.5) {};
        \node (8) at (-0.985,0.174) {};
        \node (9) at (-0.642,0.766) {};
        \draw (6) -- (7) -- (8) -- (9) -- (1) -- (4);
        \draw (5) -- (4) -- (6) -- (8) -- (1);
        \draw (4) -- (7) -- (1) -- (5) -- (9) -- (4) -- (8) -- (3) -- (7) -- (2) -- (6) -- (1);
        \draw (3) -- (6) -- (9) -- (3) -- (5) -- (7) -- (9) -- (2) -- (5) -- (8) -- (2);
    \end{tikzpicture}} & \makecell{$(K_{2,1,1} \sqcup K_2 \sqcup K_1 \sqcup K_1 \sqcup K_1)^{\textup{c}}$, *}   & \makecell{$3$}  & \makecell{$\infty$} \\
    
    & \makecell{\begin{tikzpicture}[scale=0.55,every node/.style={circle,fill=black,inner sep=2pt}]
        \node (1) at (0,1) {};
        \node (2) at (0.642,0.766) {};
        \node (3) at (0.985,0.174) {};
        \node (4) at (0.866,-0.5) {};
        \node (5) at (0.342,-0.940) {};
        \node (6) at (-0.342,-0.940) {};
        \node (7) at (-0.866,-0.5) {};
        \node (8) at (-0.985,0.174) {};
        \node (9) at (-0.642,0.766) {};
        \draw (6) -- (7) -- (8) -- (9) -- (1) -- (4);
        \draw (6) -- (5) -- (4) -- (6) -- (8) -- (1);
        \draw (4) -- (7) -- (1) -- (5) -- (9) -- (4) -- (8) -- (3) -- (7) -- (2) -- (6) -- (1);
        \draw (3) -- (6) -- (9) -- (3) -- (5) -- (7) -- (9) -- (2) -- (5) -- (8) -- (2);
    \end{tikzpicture}} & \makecell{$(K_{2,1,1} \sqcup K_1 \sqcup K_1 \sqcup K_1 \sqcup K_1 \sqcup K_1)^{\textup{c}}$} & \makecell{$2$} &  \makecell{$\infty$} \\
    
    & \makecell{\begin{tikzpicture}[scale=0.55,every node/.style={circle,fill=black,inner sep=2pt}]
        \node (1) at (0,1) {};
        \node (2) at (0.642,0.766) {};
        \node (3) at (0.985,0.174) {};
        \node (4) at (0.866,-0.5) {};
        \node (5) at (0.342,-0.940) {};
        \node (6) at (-0.342,-0.940) {};
        \node (7) at (-0.866,-0.5) {};
        \node (8) at (-0.985,0.174) {};
        \node (9) at (-0.642,0.766) {};
        \draw (3) -- (5) -- (7);
        \draw (6) -- (8) -- (1) -- (4) -- (7);
        \draw (3) -- (6) -- (9) -- (1) -- (5) -- (9) -- (2) -- (4) -- (8) -- (3) -- (7) -- (2) -- (6);
    \end{tikzpicture}} & \makecell{$K_3 \mathbin{\square} K_3$} & \makecell{$2$} &  \makecell{$\infty$} \\
    
    & \makecell{\begin{tikzpicture}[scale=0.55,every node/.style={circle,fill=black,inner sep=2pt}]
        \node (1) at (0,1) {};
        \node (2) at (0.642,0.766) {};
        \node (3) at (0.985,0.174) {};
        \node (4) at (0.866,-0.5) {};
        \node (5) at (0.342,-0.940) {};
        \node (6) at (-0.342,-0.940) {};
        \node (7) at (-0.866,-0.5) {};
        \node (8) at (-0.985,0.174) {};
        \node (9) at (-0.642,0.766) {};
        \draw (3) -- (6) -- (9) -- (3) -- (4);
        \draw (6) -- (7);
        \draw (9) -- (1);
        \draw (3) -- (5) -- (7);
        \draw (9) -- (2) -- (4);
        \draw (6) -- (8) -- (1) -- (4) -- (7) -- (1) -- (5) -- (9) -- (4) -- (8) -- (3) -- (7) -- (2) -- (6) -- (1);
        \draw (2) -- (5) -- (8) -- (2);
    \end{tikzpicture}} & \makecell{$K_{3,3,3}$} & \makecell{$2$} &   \makecell{$\infty$} \\
    
    & \makecell{\begin{tikzpicture}[scale=0.55,every node/.style={circle,fill=black,inner sep=2pt}]
        \node (1) at (0,1) {};
        \node (2) at (0.642,0.766) {};
        \node (3) at (0.985,0.174) {};
        \node (4) at (0.866,-0.5) {};
        \node (5) at (0.342,-0.940) {};
        \node (6) at (-0.342,-0.940) {};
        \node (7) at (-0.866,-0.5) {};
        \node (8) at (-0.985,0.174) {};
        \node (9) at (-0.642,0.766) {};
        \draw (3) -- (6) -- (9) -- (3) -- (4);
        \draw (6) -- (7) -- (8) -- (9) -- (7);
        \draw (9) -- (1);
        \draw (3) -- (5) -- (7);
        \draw (9) -- (2) -- (4);
        \draw (6) -- (8) -- (1) -- (4) -- (7) -- (1) -- (5) -- (9) -- (4) -- (8) -- (3) -- (7) -- (2) -- (6) -- (1);
        \draw (2) -- (5) -- (8) -- (2);
    \end{tikzpicture}} & \makecell{$K_{3,3,1,1,1}$} & \makecell{$2$} &  \makecell{$\infty$} \\
    
    & \makecell{\begin{tikzpicture}[scale=0.55,every node/.style={circle,fill=black,inner sep=2pt}]
        \node (1) at (0,1) {};
        \node (2) at (0.642,0.766) {};
        \node (3) at (0.985,0.174) {};
        \node (4) at (0.866,-0.5) {};
        \node (5) at (0.342,-0.940) {};
        \node (6) at (-0.342,-0.940) {};
        \node (7) at (-0.866,-0.5) {};
        \node (8) at (-0.985,0.174) {};
        \node (9) at (-0.642,0.766) {};
        \draw (3) -- (4);
        \draw (5) -- (6);
        \draw (7) -- (8) -- (9) -- (1);
        \draw (3) -- (5) -- (7) -- (9) -- (2) -- (4) -- (6) -- (8) -- (1) -- (4) -- (7) -- (1) -- (5) -- (9) -- (4) -- (8) -- (3) -- (7) -- (2) -- (6) -- (1);
        \draw (2) -- (5) -- (8) -- (2);
        \draw (3) -- (6) -- (9) -- (3);
    \end{tikzpicture}} & \makecell{$K_{3,2,2,1,1}$, *}  & \makecell{$2$} & \makecell{$\infty$} \\
    
    & \makecell{\begin{tikzpicture}[scale=0.55,every node/.style={circle,fill=black,inner sep=2pt}]
        \node (1) at (0,1) {};
        \node (2) at (0.642,0.766) {};
        \node (3) at (0.985,0.174) {};
        \node (4) at (0.866,-0.5) {};
        \node (5) at (0.342,-0.940) {};
        \node (6) at (-0.342,-0.940) {};
        \node (7) at (-0.866,-0.5) {};
        \node (8) at (-0.985,0.174) {};
        \node (9) at (-0.642,0.766) {};
        \draw (3) -- (4);
        \draw (5) -- (6) -- (7) -- (8) -- (9) -- (1);
        \draw (3) -- (5) -- (7) -- (9) -- (2) -- (4) -- (6) -- (8) -- (1) -- (4) -- (7) -- (1) -- (5) -- (9) -- (4) -- (8) -- (3) -- (7) -- (2) -- (6) -- (1);
        \draw (2) -- (5) -- (8) -- (2);
        \draw (3) -- (6) -- (9) -- (3);
    \end{tikzpicture}} & \makecell{$K_{3,2,1,1,1,1}$, *} & \makecell{$2$} &   \makecell{$\infty$} \\
    
    & \makecell{\begin{tikzpicture}[scale=0.55,every node/.style={circle,fill=black,inner sep=2pt}]
        \node (1) at (0,1) {};
        \node (2) at (0.642,0.766) {};
        \node (3) at (0.985,0.174) {};
        \node (4) at (0.866,-0.5) {};
        \node (5) at (0.342,-0.940) {};
        \node (6) at (-0.342,-0.940) {};
        \node (7) at (-0.866,-0.5) {};
        \node (8) at (-0.985,0.174) {};
        \node (9) at (-0.642,0.766) {};
        \draw (3) -- (4) -- (5) -- (6) -- (7) -- (8) -- (9) -- (1);
        \draw (3) -- (5) -- (7) -- (9) -- (2) -- (4) -- (6) -- (8) -- (1) -- (4) -- (7) -- (1) -- (5) -- (9) -- (4) -- (8) -- (3) -- (7) -- (2) -- (6) -- (1);
        \draw (2) -- (5) -- (8) -- (2);
        \draw (3) -- (6) -- (9) -- (3);
    \end{tikzpicture}} & \makecell{$K_{3,1,1,1,1,1,1}, *$} & \makecell{$2$} &  \makecell{$\infty$} \\
    
    & \makecell{\begin{tikzpicture}[scale=0.55,every node/.style={circle,fill=black,inner sep=2pt}]
        \node (1) at (0,1) {};
        \node (2) at (0.642,0.766) {};
        \node (3) at (0.985,0.174) {};
        \node (4) at (0.866,-0.5) {};
        \node (5) at (0.342,-0.940) {};
        \node (6) at (-0.342,-0.940) {};
        \node (7) at (-0.866,-0.5) {};
        \node (8) at (-0.985,0.174) {};
        \node (9) at (-0.642,0.766) {};
        \draw (4) -- (5);
        \draw (6) -- (7) -- (8) -- (9) -- (1) -- (3) -- (5) -- (7) -- (9) -- (2) -- (4) -- (6) -- (8) -- (1);
        \draw (4) -- (7) -- (1) -- (5) -- (9) -- (4) -- (8) -- (3) -- (7) -- (2) -- (6) -- (1);
        \draw (2) -- (5) -- (8) -- (2);
        \draw (3) -- (6) -- (9) -- (3);
    \end{tikzpicture}} & \makecell{$(K_{2,2} \sqcup K_2 \sqcup K_1 \sqcup K_1 \sqcup K_1)^{\textup{c}}$, *} & \makecell{$3$}   & \makecell{$\infty$} \\
    
    & \makecell{\begin{tikzpicture}[scale=0.55,every node/.style={circle,fill=black,inner sep=2pt}]
        \node (1) at (0,1) {};
        \node (2) at (0.642,0.766) {};
        \node (3) at (0.985,0.174) {};
        \node (4) at (0.866,-0.5) {};
        \node (5) at (0.342,-0.940) {};
        \node (6) at (-0.342,-0.940) {};
        \node (7) at (-0.866,-0.5) {};
        \node (8) at (-0.985,0.174) {};
        \node (9) at (-0.642,0.766) {};
        \draw (4) -- (5) -- (6) -- (7) -- (8) -- (9) -- (1) -- (3) -- (5) -- (7) -- (9) -- (2) -- (4) -- (6) -- (8) -- (1);
        \draw (4) -- (7) -- (1) -- (5) -- (9) -- (4) -- (8) -- (3) -- (7) -- (2) -- (6) -- (1);
        \draw (2) -- (5) -- (8) -- (2);
        \draw (3) -- (6) -- (9) -- (3);
    \end{tikzpicture}} & \makecell{$(K_{2,2} \sqcup K_1 \sqcup K_1 \sqcup K_1 \sqcup K_1 \sqcup K_1)^{\textup{c}}$} & \makecell{$2$} & \makecell{$\infty$} \\
    
    & \makecell{\begin{tikzpicture}[scale=0.55,every node/.style={circle,fill=black,inner sep=2pt}]
        \node (1) at (0,1) {};
        \node (2) at (0.642,0.766) {};
        \node (3) at (0.985,0.174) {};
        \node (4) at (0.866,-0.5) {};
        \node (5) at (0.342,-0.940) {};
        \node (6) at (-0.342,-0.940) {};
        \node (7) at (-0.866,-0.5) {};
        \node (8) at (-0.985,0.174) {};
        \node (9) at (-0.642,0.766) {};
        \draw (1) -- (2);
        \draw (3) -- (4) -- (5);
        \draw (6) -- (7) -- (8);
        \draw (9) -- (1) -- (3) -- (5) -- (7) -- (9) -- (2) -- (4) -- (6) -- (8) -- (1) -- (4) -- (7) -- (1) -- (5) -- (9) -- (4) -- (8) -- (3) -- (7) -- (2) -- (6) -- (1);
        \draw (2) -- (5) -- (8) -- (2);
        \draw (3) -- (6) -- (9) -- (3);
    \end{tikzpicture}} & \makecell{$K_{2,2,2,1,1,1}$}  & \makecell{$2$}  & \makecell{$\infty$} \\
    
    & \makecell{\begin{tikzpicture}[scale=0.55,every node/.style={circle,fill=black,inner sep=2pt}]
        \node (1) at (0,1) {};
        \node (2) at (0.642,0.766) {};
        \node (3) at (0.985,0.174) {};
        \node (4) at (0.866,-0.5) {};
        \node (5) at (0.342,-0.940) {};
        \node (6) at (-0.342,-0.940) {};
        \node (7) at (-0.866,-0.5) {};
        \node (8) at (-0.985,0.174) {};
        \node (9) at (-0.642,0.766) {};
        \draw (1) -- (2) -- (3);
        \draw (4) -- (5) -- (6) -- (7);
        \draw (8) -- (9) -- (1) -- (3) -- (5) -- (7) -- (9) -- (2) -- (4) -- (6) -- (8) -- (1) -- (4) -- (7) -- (1) -- (5) -- (9) -- (4) -- (8) -- (3) -- (7) -- (2) -- (6) -- (1);
        \draw (2) -- (5) -- (8) -- (2);
        \draw (3) -- (6) -- (9) -- (3);
    \end{tikzpicture}} & \makecell{$K_{2,2,1,1,1,1,1}$} & \makecell{$2$}  & \makecell{$\infty$} \\
    
    & \makecell{\begin{tikzpicture}[scale=0.55,every node/.style={circle,fill=black,inner sep=2pt}]
        \node (1) at (0,1) {};
        \node (2) at (0.642,0.766) {};
        \node (3) at (0.985,0.174) {};
        \node (4) at (0.866,-0.5) {};
        \node (5) at (0.342,-0.940) {};
        \node (6) at (-0.342,-0.940) {};
        \node (7) at (-0.866,-0.5) {};
        \node (8) at (-0.985,0.174) {};
        \node (9) at (-0.642,0.766) {};
        \draw (1) -- (2) -- (3) -- (4) -- (5);
        \draw (6) -- (7) -- (8) -- (9) -- (1) -- (3) -- (5) -- (7) -- (9) -- (2) -- (4) -- (6) -- (8) -- (1) -- (4) -- (7) -- (1) -- (5) -- (9) -- (4) -- (8) -- (3) -- (7) -- (2) -- (6) -- (1);
        \draw (2) -- (5) -- (8) -- (2);
        \draw (3) -- (6) -- (9) -- (3);
    \end{tikzpicture}} & \makecell{$K_{2,1,1,1,1,1,1,1}$} & \makecell{$2$} & \makecell{$\infty$} \\
    
    & \makecell{\begin{tikzpicture}[scale=0.55,every node/.style={circle,fill=black,inner sep=2pt}]
        \node (1) at (0,1) {};
        \node (2) at (0.642,0.766) {};
        \node (3) at (0.985,0.174) {};
        \node (4) at (0.866,-0.5) {};
        \node (5) at (0.342,-0.940) {};
        \node (6) at (-0.342,-0.940) {};
        \node (7) at (-0.866,-0.5) {};
        \node (8) at (-0.985,0.174) {};
        \node (9) at (-0.642,0.766) {};
        \draw (1) -- (2) -- (3) -- (4) -- (5) -- (6) -- (7) -- (8) -- (9) -- (1) -- (3) -- (5) -- (7) -- (9) -- (2) -- (4) -- (6) -- (8) -- (1) -- (4) -- (7) -- (1) -- (5) -- (9) -- (4) -- (8) -- (3) -- (7) -- (2) -- (6) -- (1);
        \draw (2) -- (5) -- (8) -- (2);
        \draw (3) -- (6) -- (9) -- (3);
    \end{tikzpicture}} & \makecell{$K_9$} & \makecell{$2$} &  \makecell{$\infty$} \\\bottomrule
    \caption{The bandwidth of all simple connected graphs on $1 \leq n \leq 9$ vertices that are $\{-1,0,1\}$-diagonalizable. Hadamard-diagonalizable graphs are those with $\{-1,1\}$-bandwidth equal to $1$. Weakly Hadamard-diagonalizable graphs are those with $\{-1,0,1\}$-bandwidth at most $2$.  An asterisk in the ``Name'' column is used to identify the $\{-1,0,1\}$-diagonalizable graphs that have $\{-1,0,1\}$-bandwidth strictly greater than $2$ or have $\{-1,0,1\}$-bandwidth at most 2 but were missed in the computer search performed in \cite{adm2021weakly}.}
    \label{tab:small_graphs}
\end{longtable}

It is interesting that every single graph on $9$ or fewer vertices that is $\{-1,0,1\}$-diagonalizable can be constructed using $K_1$ via the operations and theorems described earlier in the paper:
\begin{itemize}
    \item complementation (Proposition~\ref{prop:neg101properties}(a) and Theorem~\ref{thm:complement_of_neg_zero_one});
    
    \item disjoint union (Proposition~\ref{prop:neg101properties}(b)); and

    \item the Cartesian product (Proposition~\ref{prop:neg101properties}(c)).
\end{itemize}

\section{Generalization to sets other than $\{-1,0,1\}$ and $\{-1,1\}$}\label{sec:other_set_S}

Many of our definitions and results do not actually require that we work with the sets $\{-1,0,1\}$ or $\{-1,1\}$; we were simply motivated to study these sets based on recent work in the area. In this section, we briefly clarify which of our results can generalized to diagonalizations involving other sets of numbers.

\subsection{Generalizations of definitions}\label{sec:other_set_S_defn}

To start, we generalize Definition~\ref{defn:S_diag} as follows:

\begin{definition}\label{defn:S_diagS}
    Given a set $S \subseteq \C$, we say that a graph $G$ is \emph{$S$-diagonalizable} if there is a basis of eigenvectors for the Laplacian $L$ of $G$ whose entries all belong to $S$. Equivalently, $G$ is $S$-diagonalizable if there exists a matrix $P$, whose entries all belong to $S$, with the property that $P^{-1}LP$ is diagonal.
\end{definition}

We considered the cases $S = \{-1,0,1\}$ and $S = \{-1,1\}$ exclusively in the earlier sections of the paper. However, we also showed just prior to Definition~\ref{defn:S_diag} that every Laplacian integral graph is $S$-diagonalizable when $S = \mathbb{Z}$. 

We can also generalize the definition of a graph's $\{-1,0,1\}$- or $\{-1,1\}$-bandwidth (Definition~\ref{def:bandwidth}) as follows:

\begin{definition}\label{def:bandwidthS}
    Let $S \subseteq \C$ and let $G$ be a graph with Laplacian matrix $L$. 
    \begin{enumerate}
        \item[a)] If $G$ is not $S$-diagonalizable then the \emph{$S$-bandwidth} of $G$ is $\infty$.
        
        \item[b)] If $G$ is $S$-diagonalizable then the \emph{$S$-bandwidth} of $G$ is the smallest $k \in \mathbb{Z}_+$ for which there exists a matrix $P$ with all of the following properties:
        \begin{itemize}
            \item all of its entries belong to $S$;

            \item $P^{-1}LP$ is diagonal; and

            \item $P^TP$ has bandwidth $k$.
        \end{itemize}
    \end{enumerate}
\end{definition}

These notions of $S$-diagonalizability and $S$-bandwidth are not restricted to the real numbers. For example, taking $S$ to be the unit circle in the complex plane results in the graphs with $S$-bandwidth equal to $1$ being exactly the complex Hadamard diagonalizable graphs \cite{chan2020}. In this way, Definition~\ref{defn:S_diagS} includes several recently-investigated families of graphs as special cases.

Finally, we can also generalize our definition of balanced vectors (Definition~\ref{defn:totally_balanced}) to sets other than $\{-1,0,1\}$:

\begin{definition}\label{defn:totally_balancedS}
    Let $S \subseteq \C$. A vector $\mathbf{v} \in \C^p$ ($p \geq 2$) is \emph{$S$-balanced} if there exists a matrix $A \in \mathcal{M}_{p-1,p}$, all of whose entries belong to $S$, with $\mathrm{null}(A) = \mathrm{span}(\mathbf{v})$.
\end{definition}

For example, the vector $(2,1,1,1,1)$ is $\{-1,1\}$-balanced, as evidenced by the $\{-1,1\}$-matrix
\[
    A = \begin{bmatrix}
        1 & 1 & -1 & -1 & -1 \\
        1 & -1 & 1 & -1 & -1 \\
        1 & -1 & -1 & 1 & -1 \\
        1 & -1 & -1 & -1 & 1
    \end{bmatrix},
\]
which has $\mathrm{null}(A) = \mathrm{span}\big((2,1,1,1,1)\big)$.

\subsection{Generalizations of results}\label{sec:other_set_S_res}

To start, we generalize Proposition~\ref{prop:neg101properties} to arbitrary sets $S$ as follows:

\begin{proposition}\label{prop:neg101propertiesS}
    Let $S \subseteq \C$, and let $G$ and $H$ be graphs with $m$ and $n$ vertices, respectively.
   \begin{enumerate}
       \item[a)] If $G$ is connected and $S$-diagonalizable with bandwidth $k$ then its complement $G^{\textup{c}}$ is $S$-diagonalizable with bandwidth at most $k$.
       
       \item[b)] If $0 \in S$ and $G$ and $H$ are $S$-diagonalizable with bandwidths $k$ and $\ell$, respectively, then the disjoint union $G \sqcup H$ is $S$-diagonalizable with bandwidth $\max\{k,\ell\}$.
       
       \item[c)] If $S$ is closed under multiplication and $G$ and $H$ are $S$-diagonalizable with bandwidths $k$ and $\ell$, respectively, then their Cartesian product $G \mathbin{\square} H$ is $S$-diagonalizable with bandwidth at most $\min\{n(k-1)+\ell,m(\ell-1)+k\}$.
   \end{enumerate} 
\end{proposition}

The above proposition can be proved in an identical fashion to how analogous results were proved in \cite{adm2021weakly}; we just provide the generalized proofs here for completeness.

\begin{proof}[Proof of Proposition~\ref{prop:neg101propertiesS}]
    To prove~(a), let $L$ be the Laplacian of $G$ and note that the Laplacian of $G^{\textup{c}}$ is $(nI - \mathbf{1}\mathbf{1}^T) - L$. Since $G$ is connected, all eigenvectors of $L$ that are not multiples of $\mathbf{1}$ are orthogonal to $\mathbf{1}$. It follows that any matrix $P$ that diagonalizes $L$ also diagonalizes $(nI - \mathbf{1}\mathbf{1}^T) - L$.

    To prove~(b), simply notice that if $P$ and $Q$ diagonalize the Laplacians of $G$ and $H$, respectively, then the block matrix
    \begin{align}\label{eq:block_diag_with_zeros}
        \begin{bmatrix}
            P & O \\ O & Q
        \end{bmatrix}
    \end{align}
    diagonalizes the Laplacian of $G \sqcup H$.

    Finally, to prove~(c), we just note that if $P$ and $Q$ diagonalize the Laplacians of $G$ and $H$, respectively, then
    \[
        P \otimes Q
    \]
    diagonalizes the Laplacian of $G \mathbin{\square} H$.
\end{proof}

In part~(b), the hypothesis $0 \in S$ is required because the columns of the block diagonal matrix~\eqref{eq:block_diag_with_zeros} has zeroes in it. This is necessary; we noted earlier that property~(b) fails even when $S = \{-1,1\}$. In property~(c), it is perhaps worth noting that the restriction that $S$ is closed under multiplication is satisfied, for example, by the sets $S = \{-1,0,1\}$, $S = \{-1,1\}$, and $S = \{z \in \C : |z| = 1\}$, so this property applies to all cases of interest that we are aware of.

The only other result from before Section~\ref{sec:complements} that generalizes to other sets $S$ in any meaningful way is Proposition~\ref{prop:tot_bal_rescale}, which still works as long as $S \subseteq \mathbb{Z}$.

However, our main results concerning graph complements and joins in Section~\ref{sec:complements} do generalize straightforwardly to other sets $S$. In particular, Theorem~\ref{thm:complement_of_neg_zero_one} and Corollary~\ref{cor:join_neg_zero_one} can be generalized from $\{-1,0,1\}$-diagonalizability and balanced vectors to $S$-diagonalizability and $S$-balanced vectors with no substantial changes to their proofs:

\begin{theorem}\label{thm:complement_of_neg_zero_oneS}
    Let $S \subseteq \C$ and suppose that $G$ is an $S$-diagonalizable graph with $p$ connected components of sizes $v_1$, $v_2$, $\ldots$, $v_p$, respectively. Then $G^{\textup{c}}$ is $S$-diagonalizable if and only if the vector $(v_1,v_2,\ldots,v_p)$ is $S$-balanced. Furthermore, if $G$ has $S$-bandwidth $k$ then $G^{\textup{c}}$ has $S$-bandwidth at most $\max\{p-1,k\}$.
\end{theorem}

\begin{corollary}\label{cor:join_neg_zero_oneS}
    Let $S \subseteq \C$ have $0 \in S$ and let $p \in \mathbb{Z}_{+}$. Suppose that, for $1 \leq j \leq p$, $G_j$ is a connected $v_j$-vertex $S$-diagonalizable graph with $S$-bandwidth $k_j$. Then the join
    \[
        G_1^{\textup{c}} \vee G_2^{\textup{c}} \vee \cdots \vee G_p^{\textup{c}}
    \]
    is $S$-diagonalizable if and only if the vector $(v_1, v_2, \ldots, v_p)$ is $S$-balanced. Furthermore, it has $S$-bandwidth at most $\max\{p-1,k_1,k_2,\ldots,k_p\}$.
\end{corollary}

It is worth noting that the ``$0 \in S$'' hypothesis of Corollary~\ref{cor:join_neg_zero_oneS} really is necessary (which should not be surprising, since the proof of this corollary makes use of Proposition~\ref{prop:neg101propertiesS}(b), which requires $0 \in S$). For example, we noted earlier that the vector $(2,1,1,1,1)$ is $\{-1,1\}$-balanced, and it is also straightforward to show that $K_2$ and $K_1$ are each $\{-1,1\}$-diagonalizable. However, the graph
\[
    K_2^{\textup{c}} \vee K_1^{\textup{c}} \vee K_1^{\textup{c}} \vee K_1^{\textup{c}} \vee K_1^{\textup{c}} = K_{2,1,1,1,1}
\]
is not $\{-1,1\}$-diagonalizable, since the eigenspace corresponding to the eigenvalue $4$ of its Laplacian is $\mathrm{span}\{(1,-1,0,0,0,0)\}$, which does not contain a vector with entries from $\{-1,1\}$.

\section{Conclusion and future work}\label{sec:conclusions}

Motivated by the notion of weak Hadamard diagonalizable graphs introduced in \cite{adm2021weakly}, we studied graphs whose Laplacian matrix is diagonalized by matrices with all entries belonging to $\{-1,0,1\}$ or $\{-1,1\}$, with or without the assumption of orthogonality or quasi-orthogonality of its columns. 

Our work significantly clarifies numerous questions concerning $\{-1,0,1\}$-diagonalizability of graphs. For example, it was noted in the novel work \cite{adm2021weakly} that $K_{3,2,2,1}$ is $\{-1,0,1\}$-diagonalizable, despite none of the results therein really capturing ``why''. By contrast, our results (Corollary~\ref{cor:complete_multipartite_graph} in particular) makes this transparent: $K_{3,2,2,1}$ is $\{-1,0,1\}$-diagonalizable because the vector $(3,2,2,1)$ is balanced. However, our work has also raised numerous questions that we believe are worth exploring:

\begin{itemize}
    \item Is Conjecture~\ref{conj:hadamard_zero} true? This is a natural variant of the well-known Hadamard conjecture (i.e., Conjecture~\ref{conj:hadamard}). 

    \item All $\{-1,1\}$-diagonalizable graphs that we have found are regular. Are they indeed all regular? It was proved in \cite[Theorem~5]{BFK11} that this is true for Hadamard-diagonalizable graphs, but it is unclear whether or not the proof can be modified to work without orthogonality of the rows and/or columns of the $\{-1,1\}$-matrix.

    \item Related to the previous question, a brute-force computer search reveals that the only regular connected graph on $n = 10$ vertices that is $\{-1,1\}$-diagonalizable is the complete graph (refer back to Theorem~\ref{thm:complete_is_pm1}(b)). It thus is natural to ask whether or not the complete graph is the only connected graph that is $\{-1,1\}$-diagonalizable  when $n \equiv 2 \pmod{4}$.
\end{itemize}

\section*{Acknowledgements}
   The authors thank Karen Meagher providing insight into how the graphs in \cite[Appendix A]{adm2021weakly} were generated. The authors thank Luis Varona and Benjamin Talbot for finding and correcting some bandwidth calculation mistakes in Table~\ref{tab:small_graphs}.  N.J.\ thanks Margaret-Ellen Messinger   for helpful conversations. 
N.J.\ was supported by NSERC Discovery Grant number RGPIN-2022-04098. S.P.\ was supported by NSERC Discovery Grant number 1174582, the Canada Foundation for Innovation (CFI) grant number 35711, and the Canada Research Chairs (CRC) Program grant number 231250. 

\bibliographystyle{alpha}
\bibliography{ref}

\end{document}